\documentclass[12pt]{amsart}
\usepackage[nobysame,abbrev,alphabetic]{amsrefs}
\usepackage{amssymb}
\usepackage[all]{xy}
\usepackage{tikz-cd}
\usepackage{mathtools}
\usepackage{xcolor}

\DeclareMathOperator{\Gal}{Gal}

\DeclareMathOperator{\id}{id}
\DeclareMathOperator{\Img}{Im}

\DeclareMathOperator{\Ker}{Ker}

\DeclareMathOperator{\res}{res}

\DeclareMathOperator{\C}{C}
\DeclareMathOperator{\Ho}{H}
\DeclareMathOperator{\B}{B}
\DeclareMathOperator{\Z}{Z}
\DeclareMathOperator{\aug}{aug}
\DeclareMathOperator{\acup}{\cup_{aug}}
\DeclareMathOperator{\Pic}{Pic}
\DeclareMathOperator{\Br}{Br}
\DeclareMathOperator{\Div}{Div}
\DeclareMathOperator{\fdiv}{div}

\DeclareMathOperator{\Deg}{deg}
\DeclareMathOperator{\incl}{incl}
\DeclareMathOperator{\cores}{cor}

\DeclareMathOperator{\Ind}{Ind_G^H}

\DeclareFontFamily{U}{wncy}{}
\DeclareFontShape{U}{wncy}{m}{n}{<->wncyr10}{}
\DeclareSymbolFont{mcy}{U}{wncy}{m}{n}
\DeclareMathSymbol{\Sha}{\mathord}{mcy}{"58}
\DeclareMathSymbol{\sha}{\mathord}{mcy}{"78}

\begin{document}

\newtheorem{thm}{Theorem}[section]
\newtheorem{cor}[thm]{Corollary}
\newtheorem{lem}[thm]{Lemma}
\newtheorem{prop}[thm]{Proposition}
\newtheorem{defin}[thm]{Definition}
\newtheorem{exam}[thm]{Example}
\newtheorem{examples}[thm]{Examples}
\newtheorem{rem}[thm]{Remark}
\newtheorem{case}{\sl Case}
\newtheorem{claim}{Claim}
\newtheorem{prt}{Part}
\newtheorem*{mainthm}{Main Theorem}
\newtheorem*{thmA}{Theorem A}
\newtheorem*{thmB}{Theorem B}
\newtheorem*{thmC}{Theorem C}
\newtheorem*{thmD}{Theorem D}
\newtheorem{question}[thm]{Question}
\newtheorem*{notation}{Notation}
\swapnumbers
\newtheorem{rems}[thm]{Remarks}
\newtheorem*{acknowledgment}{Acknowledgment}

\newtheorem{questions}[thm]{Questions}
\numberwithin{equation}{section}

\newcommand{\Bock}{\mathrm{Bock}}
\newcommand{\dec}{\mathrm{dec}}
\newcommand{\diam}{\mathrm{diam}}
\newcommand{\dirlim}{\varinjlim}
\newcommand{\discup}{\ \ensuremath{\mathaccent\cdot\cup}}
\newcommand{\divis}{\mathrm{div}}
\newcommand{\gr}{\mathrm{gr}}
\newcommand{\nek}{,\ldots,}
\newcommand{\ind}{\hbox{ind}}
\newcommand{\inv}{^{-1}}
\newcommand{\isom}{\cong}
\newcommand{\Massey}{\mathrm{Massey}}
\newcommand{\ndiv}{\hbox{$\,\not|\,$}}
\newcommand{\nil}{\mathrm{nil}}
\newcommand{\pr}{\mathrm{pr}}
\newcommand{\tagg}{''}
\newcommand{\tensor}{\otimes}
\newcommand{\alp}{\alpha}
\newcommand{\gam}{\gamma}
\newcommand{\Gam}{\Gamma}
\newcommand{\del}{\delta}
\newcommand{\Del}{\Delta}
\newcommand{\eps}{\epsilon}
\newcommand{\lam}{\lambda}
\newcommand{\Lam}{\Lambda}
\newcommand{\sig}{\sigma}
\newcommand{\Sig}{\Sigma}
\newcommand{\bfA}{\mathbf{A}}
\newcommand{\bfB}{\mathbf{B}}
\newcommand{\bfC}{\mathbf{C}}
\newcommand{\bfF}{\mathbf{F}}
\newcommand{\bfP}{\mathbf{P}}
\newcommand{\bfQ}{\mathbf{Q}}
\newcommand{\bfR}{\mathbf{R}}
\newcommand{\bfS}{\mathbf{S}}
\newcommand{\bfT}{\mathbf{T}}
\newcommand{\bfZ}{\mathbf{Z}}
\newcommand{\dbA}{\mathbb{A}}
\newcommand{\dbC}{\mathbb{C}}
\newcommand{\dbF}{\mathbb{F}}
\newcommand{\dbN}{\mathbb{N}}
\newcommand{\dbQ}{\mathbb{Q}}
\newcommand{\dbR}{\mathbb{R}}
\newcommand{\dbU}{\mathbb{U}}
\newcommand{\dbZ}{\mathbb{Z}}
\newcommand{\grf}{\mathfrak{f}}
\newcommand{\gra}{\mathfrak{a}}
\newcommand{\grA}{\mathfrak{A}}
\newcommand{\grB}{\mathfrak{B}}
\newcommand{\grd}{\mathfrak{d}}
\newcommand{\grh}{\mathfrak{h}}
\newcommand{\grI}{\mathfrak{I}}
\newcommand{\grL}{\mathfrak{L}}
\newcommand{\grm}{\mathfrak{m}}
\newcommand{\grp}{\mathfrak{p}}
\newcommand{\grq}{\mathfrak{q}}
\newcommand{\grR}{\mathfrak{R}}
\newcommand{\calA}{\mathcal{A}}
\newcommand{\calB}{\mathcal{B}}
\newcommand{\calC}{\mathcal{C}}
\newcommand{\calE}{\mathcal{E}}
\newcommand{\calG}{\mathcal{G}}
\newcommand{\calH}{\mathcal{H}}
\newcommand{\calK}{\mathcal{K}}
\newcommand{\calL}{\mathcal{L}}
\newcommand{\calM}{\mathcal{M}}
\newcommand{\calW}{\mathcal{W}}
\newcommand{\calV}{\mathcal{V}}

\makeatletter
\renewcommand{\BibLabel}{%
    \Hy@raisedlink{\hyper@anchorstart{cite.\CurrentBib}\hyper@anchorend}%
    [\thebib]%
}

\setcounter{tocdepth}{3}

\renewcommand{\tocsection}[3]{%
  \indentlabel{\@ifnotempty{#2}{\bfseries\ignorespaces#1 #2\quad}}\bfseries#3}
\renewcommand{\tocsubsection}[3]{%
  \indentlabel{\@ifnotempty{#2}{\ignorespaces#1 #2\quad}}#3}

\newcommand\@dotsep{4.5}
\def\@tocline#1#2#3#4#5#6#7{\relax
  \ifnum #1>\c@tocdepth 
  \else
    \par \addpenalty\@secpenalty\addvspace{#2}%
    \begingroup \hyphenpenalty\@M
    \@ifempty{#4}{%
      \@tempdima\csname r@tocindent\number#1\endcsname\relax
    }{%
      \@tempdima#4\relax
    }%
    \parindent\z@ \leftskip#3\relax \advance\leftskip\@tempdima\relax
    \rightskip\@pnumwidth plus1em \parfillskip-\@pnumwidth
    #5\leavevmode\hskip-\@tempdima{#6}\nobreak
    \leaders\hbox{$\m@th\mkern \@dotsep mu\hbox{.}\mkern \@dotsep mu$}\hfill
    \nobreak
    \hbox to\@pnumwidth{\@tocpagenum{\ifnum#1=1\bfseries\fi#7}}\par
    \nobreak
    \endgroup
  \fi}

\def\l@subsection{\@tocline{2}{0pt}{2.5pc}{5pc}{}}

\makeatother

\title[Augmented cup products]{AUGMENTED CUP PRODUCTS}

\author{Dor Amzaleg}




\maketitle

\begin{abstract}
In this paper, we present Tate's theory of \textbf{augmented cup products} in profinite cohomology in a modern constructive style. As an application, we interpret pairings between groups associated to curves constructed by Lichtenbaum, in terms of augmented cup products. 
\end{abstract}

\section{Introduction}
\label{Introduction}
The notion of augmented cup products was introduced by Tate in the early days of Galois cohomology (in the 1950's). These are $G$-bilinear maps between cohomology groups. The augmented cup products have similar properties to those of the standard cup product, but they are adjusted and especially useful in situations where one has reciprocity laws, as will be described below. 

Recall that given a profinite group $G$, non-negative integers $r,s$, and a $G$-bilinear pairing of discrete $G$-modules $A\times B \to C$, the \textbf{cup product} is a bilinear map 
\[
\cup : \Ho^r(G,A)\times \Ho^s(G,B)\to \Ho^{r+s}(G,C).
\]
The \textbf{augmented cup products} are bilinear maps 
\[
\acup : \Ho^r(G,A'')\times \Ho^s(G,B'')\to \Ho^{r+s+1}(G,C),
\]
where $A'',B'',C$ come from a Tate product $(A,B,C)$. Here a \textbf{Tate product} $(A,B,C)$ is a data consisting of two short exact sequences of discrete $G$-modules 
\[
    0 \xrightarrow{} A'\xrightarrow{} A\xrightarrow{} A''\xrightarrow{} 0, \]\[
    0 \xrightarrow{} B'\xrightarrow{} B\xrightarrow{} B''\xrightarrow{} 0, 
\]
and two $G$-bilinear maps
\[
    A'\times B\to C,   \quad\quad    A\times B'\to C,
\]
coinciding on $A'\times B'$. 

The reciprocity mentioned above is expressed in the following example of a Tate product. Suppose that we have an embedding of discrete $G$-modules $A'\xhookrightarrow{i} A$, and a $G$-bilinear map of discrete $G$-modules 
\[
\circ: A'\times A\to C.
\]
Suppose that for every $f,g \in A'$ there is a reciprocity law
\[
f\circ i(g) = g \circ i(f).
\]
Then $(A,A,C)$ forms a Tate product, which gives rise for every $r,s$ to an augmented cup product
\[
\acup : \Ho^r(G,A/A')\times \Ho^s(G,A/A')\to \Ho^{r+s+1}(G,C).
\]

An example of such reciprocity is Weil reciprocity law \cite{Lang83}*{Page 172}: Consider a curve over a separably closed field, and elements $f,g$ in its function field. Let $\fdiv$ be the divisor map mapping each function to its divisor. Suppose that the associated principal divisors $\fdiv(f)$ and $\fdiv(g)$ have disjoint supports. Weil reciprocity law states that
\[
f(\fdiv(g)) = g(\fdiv(f)).
\]
This will be described in detail in Chapter \ref{Applications}, and will induce the Tate product and the augmented cup product for our geometrical applications. 

In general, in various situations with a similar reciprocity (see for example \cite{Lang}*{Chapter X, Section 3}), the theory of augmented cup products can be applied to construct pairings between cohomology groups and to study their properties.

The notion of augmented cup products was presented in a 1966 book by Lang \cite{Lang}. On the other hand, other canonical textbooks of Galois cohomology such as \cite{Serre}, \cite{Koch02}, \cite{NSW}, do not mention the augmented cup products at all. In \cite{Milne} there is a short description of the augmented cup product, but without comprehensive explanations.

In the first part of this paper we give a constructive description of the augmented cup products and develop their general theory in detail. We prove their natural properties such as functoriality, (graded) commutativity, compatibility with group action, and compatibility with other morphisms. 

In his book, Lang presents the augmented cup products using a  non-constructive approach. He defines the notion for Tate products in an abelian category equipped with an abstract $\delta$-functor to another abelian category. This machinery of $\delta$-functors and dimension shifting is the core of the proofs of the various properties of the augmented cup products in Lang's book. These proofs are based on the \textit{abstract uniqueness theorem} \cite{Lang}*{Chapter 1, Section 1}, which enables to deduce isomorphisms between cohomology groups in higher dimensions from those in dimension zero, without explicit calculations. By contrast, in this paper we use a more explicit approach. We prove the various properties of the augmented cup products using detailed calculations. For example, the compatibility of the augmented cup product with the restriction and the corestriction maps, which is proved in \cite{Lang} using the abstract uniqueness theorem, is proved here very explicitly.

The second part of this paper deals with constructions made by Lichtenbaum in his important 1969 paper \cite{Lich}. In this paper, Lichtenbaum constructs three canonical pairings between groups associated to curves over fields, and proves that over $p$-adic fields these pairings are perfect. He also mentions, in a few words and without any explanations, that one of those pairings is a special case of the augmented cup product. Our main goal in this part of the paper is to explain this statement. We apply the theory developed in the first part to interpret Lichtenbaum's pairings in the general context of augmented cup products in Galois cohomology.

More specifically, consider a proper, smooth, geometrically connected curve $X$ over a field $k$ with separable closure $\bar{k}$. Consider the extension by scalars $\bar{X}=X\otimes_{k} \bar{k}$ of $X$ to $\bar{k}$. Let $G=\Gal(\bar{k}/k)$ be the absolute Galois group of $k$. Let $\Pic(X)$ be the Picard group of $X$, $\Pic_0(X)$ the subgroup of $\Pic(X)$ consisting of the divisor classes of degree zero, and $\Br(X)$ the Brauer group of $X$. Lichtenbaum constructs pairings:
\begin{enumerate}
    \item $\Ho^0(G,\Pic_0(\bar{X}))\times \Ho^1(G,\Pic_0(\bar{X}))\to\Br(k)$,
    \item $\Pic_0(X)\times \Ho^1(G,\Pic(X))\to\Br(k)$,
    \item $\Pic(X)\times \Br(X)\to\Br(k)$,
\end{enumerate}
and proves their compatibility, using the geometrical properties of the elements involved and computations in the level of cochains.

We show that the pairing (1) is indeed an augmented cup product. Further, we give different proofs, based on the general theory of the augmented cup products, to the existence of the pairing (2), and to the compatibility of the pairing (3) with the former two pairings. We also generalize some of Lichtenbaum's results by using general cohomological arguments proved in the first part.

This paper was submitted as a Master thesis at Ben-Gurion University, under the supervision of Prof. Ido Efrat.

\medskip
\medskip

\section{Cohomological preliminaries}
\label{section about Cohomological preliminaries}

\subsection{Basic profinite cohomology}
\quad\\
We recall some basic cohomological definitions and notations. Standard references for profinite cohomology are \cite{Koch02}, \cite{Milne}, \cite{NSW}, \cite{Serre}. Let $G$ be a profinite group. Let $A$ be a discrete $G$-module, and let $r$ be a non-negative integer. 

The (non-homogeneous) $r$\textbf{-cochains} are the continuous maps $G^r\to A$. The group of all $r$-cochains is denoted by $\C^r(G,A)$.

The \textbf{coboundary maps} $d_r:\C^r(G,A)\to \C^{r+1}(G,A)$ are defined by: 
\[
\begin{split}
    d_rf(x_1\nek x_{r+1})&=x_1\cdot f(x_2\nek x_{r+1})+\\
    &\sum_{i=1}^r (-1)^i\cdot f(x_1\nek x_i\cdot x_{i+1}\nek x_{r+1})+\\
    &(-1)^{r+1}\cdot f(x_1\nek x_r).
\end{split}
\]

One has $d_{r+1}\circ d_{r}=0$, i.e, we have a complex:
\[
\C^0(G,A)\xrightarrow{d_0}\C^1(G,A)\xrightarrow{d_1}\C^2(G,A)\xrightarrow{d_2}\cdots
\]

The $r$\textbf{-cocycles} are the elements of 
\[
\Z^r(G,A)=\Ker d_r.
\]

The $r$\textbf{-coboundaries} are the elements of \[
\B^r(G,A)=\Img d_{r-1}.
\]

The $r^{th}$ \textbf{cohomology group} of $G$ (with values in $A$) is 
\[
\Ho^r(G,A)=\Z^r(G,A)/\B^r(G,A).
\]

We denote the \textbf{cohomology class} of a cocycle $f$ by $[f]$.

\medskip
\medskip

\subsection{Induced morphisms of cochains and cohomology groups}
\label{Induced morphisms of cochains and cohomology groups}
\quad \\
Let $G_1$ and $G_2$ be profinite groups. Let $A_1$ be a discrete $G_1$-module, and let $A_2$ be a discrete $G_2$-module. Suppose that we have a morphism of profinite groups $\phi:G_2\to G_1$, and a morphism of abelian groups $\psi:A_1\to A_2$ such that $(\phi,\psi)$ is a morphism of discrete modules (i.e. $\psi(\phi(g_2)\cdot a_1)=g_2\cdot \psi(a_1)$, for $a_1\in A_1, g_2\in G_2$). Then for every $r$ we get an induced morphism of cochains \cite{Koch02}*{\S3.1}:
\[
\begin{split}
    (\phi,\psi)^*: \C^r(G_1,A_1)&\to \C^r(G_2,A_2)\\
    f &\mapsto \psi\circ f\circ\phi.
\end{split}
\]
Here, $\psi\circ f\circ\phi$ denotes the map given by
\[
(\sigma_1,\ldots, \sigma_r)\mapsto \psi f(\phi(\sigma_1),\ldots,\phi(\sigma_r)).
\] 

The induced morphism $(\phi,\psi)^*$ is compatible with the coboundary maps, i.e. the following diagram commutes:
\begin{equation}
\label{compatibility of induced morphism with the coboundary maps}
    \begin{tikzcd}
    \C^r(G_1,A_1)\arrow[r,"{(\phi,\psi)^*}"]\arrow[d,"d_r"]&
    \C^r(G_2,A_2)\arrow[d,"d_r"] \\
    \C^{r+1}(G_1,A_1)\arrow[r,"{(\phi,\psi)^*}"] &
    \C^{r+1}(G_2,A_2).
    \end{tikzcd}
\end{equation}
Hence, it induces a morphism of cohomology groups 
\[
\begin{split}
    (\phi,\psi)^*: \Ho^r(G_1,A_1)&\to \Ho^r(G_2,A_2)\\
    [f] &\mapsto [\psi\circ f\circ\phi].
\end{split}
\]

In the special case $G_1=G_2=G, \phi=\id_{G}$, we abbreviate the induced morphisms by
\[
\psi_*: \C^{r}(G,A_1)\to \C^{r}(G,A_2), \qquad \psi_*: \Ho^{r}(G,A_1)\to \Ho^{r}(G,A_2).
\]

For the rest of this paper, let $G, G_1, G_2$ be profinite groups.

For the following proposition we consider a short exact sequence of discrete $G$-modules:
\[
    0 \xrightarrow{} A'\xrightarrow{i} A\xrightarrow{j} A''\xrightarrow{} 0 ,
\]
and the short exact sequence of cochains that it induces:
\[
    0 \xrightarrow{} {\rm C}^{r}(G,A')\xrightarrow{i_*} {\rm C}^{r}(G,A)\xrightarrow{j_*} {\rm C}^{r}(G,A'')\xrightarrow{} 0 .
\]
This short exact sequence implies that ${\rm C}^{r}(G,A')$ may be identified with the kernel of $j_*$.

\begin{prop}
 \label{(a) df is in A' (b)f is in A' up to coboundary}
For the setting as above:\\
(a) If $f\in {\rm C}^r(G,A)$ and $j_*f\in {\rm Z}^r(G,A'')$, then $d_rf\in {\rm C}^{r+1}(G,A')$.\\
(b) If $f\in {\rm C}^r(G,A)$ and $j_*f\in {\rm B}^r(G,A'')$, then $f\in {\rm C}^r(G,A')$ up to a coboundary, i.e. there exists:
\[
 \tilde{f}\in {\rm C}^{r-1}(G,A) \hbox{ such that } f-d_{r-1}\tilde{f}\in {\rm C}^r(G,A').
\]
\end{prop}
\begin{proof}
We have the following commutative diagram with exact rows:
\begin{center}
        \begin{tikzcd}
        0\arrow[r] & {\rm C}^{r-1}(G,A')\arrow[r,"i_*"]\arrow[d,"d_{r-1}"] & {\rm C}^{r-1}(G,A)\arrow[r,"j_*"]\arrow[d,"d_{r-1}"] & {\rm C}^{r-1}(G,A'')\arrow[r]\arrow[d,"d_{r-1}"] & 0\\
        0\arrow[r] & {\rm C}^{r}(G,A')\arrow[r,"i_*"]\arrow[d,"d_r"] & {\rm C}^{r}(G,A)\arrow[r,"j_*"]\arrow[d,"d_r"] & {\rm C}^{r}(G,A'')\arrow[r]\arrow[d,"d_r"] & 0\\
        0\arrow[r] & {\rm C}^{r+1}(G,A')\arrow[r,"i_*"] & {\rm C}^{r+1}(G,A)\arrow[r,"j_*"] & {\rm C}^{r+1}(G,A'')\arrow[r] & 0.
        \end{tikzcd}
\end{center}
(a) By the commutativity of the diagram, $j_*(d_rf)=d_r(j_*f)=0$. By the exactness of the lower row, $d_rf\in {\rm C}^{r+1}(G,A')$.\\
(b) Take $\tilde{f}''\in {\rm C}^{r-1}(G,A'')$ such that $d_{r-1}\tilde{f}''=j_*f$. Since $j_*$ is surjective, there exists $\tilde{f}\in {\rm C}^{r-1}(G,A)$ such that $j_*\tilde{f}=\tilde{f}''$. Therefore, $d_{r-1}(j_*\tilde{f})=j_*f$. By the commutativity of the diagram, $j_*(d_{r-1}\tilde{f})=j_*f$. Hence $j_*(f-d_{r-1}\tilde{f})=0$, and by the exactness of the middle row $f-d_{r-1}\tilde{f}\in {\rm C}^r(G,A')$. \qedhere
\end{proof}

\medskip
\medskip

\subsection{Cup products}
\quad\\
Suppose that we have a $G$-bilinear map of discrete $G$-modules 
\[
\begin{split}
    A\times B &\to C \\
    (a,b)&\mapsto a \times b.
\end{split}
\]
Namely, the map is bilinear and for $g\in G$ one has
\[
g\cdot (a\times b) = g\cdot a \times g\cdot b.
\]
Let $r,s$ be non-negative integers. The \textbf{cup product of cochains}
\[
\cup : \C^r(G,A)\times {\rm C}^s(G,B)\to {\rm C}^{r+s}(G,C) ,
\]
is defined by: 
\begin{equation}
    \label{Cup product between two cochains groups}
    (f\cup g)(x_1\nek x_{r+s}):= f(x_1\nek x_r) \times x_1\cdots x_r\cdot g(x_{r+1}\nek x_{r+s}),
\end{equation}
for every $f\in {\rm C}^r(G,A)$, $g\in {\rm C}^s(G,B)$.

The \textbf{"Leibnitz rule"} in cohomology \cite{NSW}*{Proposition\ 1.4.1} is
\begin{equation}
\label{Leibnitz of cup products}
d_{r+s}(f\cup g)=d_rf\cup g + (-1)^rf\cup d_sg.
\end{equation}

If $f$ and $g$ are cocycles then $f\cup g$ is a cocycle as well, since
\begin{equation}
\label{Cup product between two cocycles}
d_{r+s}(f\cup g)=0\cup g + (-1)^rf\cup 0=0+0=0 .
\end{equation}

If $f$ is a coboundary and $g$ is a cocycle, or conversely, then $f\cup g$ is a coboundary. To see this, take $\tilde{f}\in {\rm C}^{r-1}(G,A)$ such that $d_{r-1}\tilde{f}=f$, and $g\in\Z^s(G,B)$. We have,
\begin{equation}
    \label{Cup product between cocycle and coboundary}
    d_{r+s-1}(\tilde{f}\cup g)=d_{r-1}\tilde{f}\cup g + (-1)^{r-1}\tilde{f}\cup d_sg=f\cup g + (-1)^{r-1}\tilde{f}\cup 0=f\cup g ,
\end{equation}
as required. The case $f\in {\rm Z}^r(G,A), g\in {\rm B}^s(G,B)$, is proved in a similar way.

The cup product of cochains induces the \textbf{cup product of cohomology classes}
\[
\cup : {\rm H}^r(G,A)\times {\rm H}^s(G,B)\to {\rm H}^{r+s}(G,C),
\]
defined as follows:\\
Given $\alpha\in {\rm H}^r(G,A), \beta\in {\rm H}^s(G,B)$, we choose representatives $f\in {\rm Z}^r(G,A), g\in {\rm Z}^s(G,B)$ such that $\alpha=[f], \beta=[g]$, and set:
\begin{equation}
\label{Cup product between cohomology classes}
\alpha\cup\beta:= [f\cup g].
\end{equation}
It is well defined since (\ref{Cup product between two cocycles}) implies that $f\cup g$ is indeed a cocycle, and (\ref{Cup product between cocycle and coboundary}) implies that the definition is independent of the choice of the representatives.

The cup product of cohomology classes is bilinear, functorial, associative, and graded commutative, in the sense that
\[
\alpha\cup\beta = (-1)^{rs}\cdot \beta\cup\alpha
\]
\cite{Koch02}*{Theorems\ 3.25,\ 3.26,\ 3.27}.

The cup product is functorial in the following sense: Suppose that we have a $G_1$-bilinear map of discrete $G_1$-modules $A_1\times B_1\to C_1$, and a $G_2$-bilinear map of discrete $G_2$-modules $A_2\times B_2\to C_2$. Suppose that $\phi: G_2\to G_1$ is a morphism of profinite groups, and that we have morphisms
\[
\psi_A: A_1\to A_2,\quad \psi_B: B_1\to B_2, \quad \psi_C: C_1\to C_2,
\] 
such that $(\phi,\psi_A), (\phi,\psi_B), (\phi,\psi_C)
$ 
are morphisms of discrete modules. Suppose that the above morphisms of discrete modules preserve bilinear maps. Namely, that the diagram
\begin{equation}
    \label{(phi,psi) preserves bilinear maps}
    \begin{tikzcd}
    A_1  \qquad
    \times \qquad 
    B_1 \quad \arrow[d, "\psi_A", xshift=-8ex] \arrow[d, "\psi_B", xshift=6ex] \arrow[r] &
    \quad C_1 \arrow[d, "\psi_C"] \\
    A_2 \qquad \times \qquad B_2 \quad \arrow[r] & \quad C_2 
    \end{tikzcd}
\end{equation}
is commutative.

\begin{prop}
\label{functoriality of the cup product}
 In the above setting:\\
(a)\quad The diagram
\[
    \begin{tikzcd}
    \C^{r}(G_1,A_1) \arrow[d, "{(\phi,\psi_A)^*}"] &
    \times &
    \C^{s}(G_1,B_1) \quad \arrow[d, "{(\phi,\psi_B)^*}"] \arrow[r, "\cup"] &
    \quad \C^{r+s}(G_1,C_1) \arrow[d, "{(\phi,\psi_C)^*}"] \\
    \C^{r}(G_2,A_2) & \times & \C^{s}(G_2,B_2) \quad \arrow[r, "\cup"] & \quad \C^{r+s}(G_2,C_2) 
    \end{tikzcd}
\]
is commutative. \\
(b)\quad The diagram
\[
\begin{tikzcd}
    \Ho^{r}(G_1,A_1) \arrow[d, "{(\phi,\psi_A)^*}"] &
    \times &
    \Ho^{s}(G_1,B_1) \quad \arrow[d, "{(\phi,\psi_B)^*}"] \arrow[r, "\cup"] &
    \quad \Ho^{r+s}(G_1,C_1) \arrow[d, "{(\phi,\psi_C)^*}"] \\
    \Ho^{r}(G_2,A_2) & \times & \Ho^{s}(G_2,B_2) \quad \arrow[r, "\cup"] & \quad \Ho^{r+s}(G_2,C_2) 
\end{tikzcd}
\]
is commutative.
\end{prop}

\begin{proof}
See \cite{NSW}*{Proposition 1.4.2}.
\end{proof}

Suppose that in addition to the setting above we have morphisms \[
\chi: G_1\to G_2, \quad \xi: B_2\to B_1,
\]
such that $(\chi,\xi)$ is a morphism of discrete modules. Suppose that $\chi\circ\phi=\id_{G_2}$, and that the diagram
\begin{equation}
    \label{twisted preserving of bilinear maps}
    \begin{tikzcd}
    A_1  \qquad
    \times \qquad 
    B_1 \quad \arrow[d, "\psi_A", xshift=-8ex] \arrow[r] &
    \quad C_1 \arrow[d, "\psi_C"] \\
    A_2 \qquad \times \qquad B_2 \quad  \arrow[u, "\xi", xshift=6ex] \arrow[r] & \quad C_2 
    \end{tikzcd}
    \end{equation}
is commutative.
\begin{prop}
\label{twisted functoriality of the cup product}
In the above setting:\\
(a)\quad The following diagram commutes:
\[
\begin{tikzcd}
    \C^{r}(G_1,A_1) \arrow[d, "{(\phi,\psi_A)^*}"] &
    \times &
    \C^{s}(G_1,B_1) \quad  \arrow[r, "\cup"] &
    \quad \C^{r+s}(G_1,C_1) \arrow[d, "{(\phi,\psi_C)^*}"] \\
    \C^{r}(G_2,A_2) & \times & \C^{s}(G_2,B_2) \arrow[u, "{(\chi,\xi)^*}"] \quad \arrow[r, "\cup"] & \quad \C^{r+s}(G_2,C_2) .
\end{tikzcd}
\]
(b)\quad The following diagram commutes:
\[
\begin{tikzcd}
    \Ho^{r}(G_1,A_1) \arrow[d, "{(\phi,\psi_A)^*}"] &
    \times &
    \Ho^{s}(G_1,B_1) \quad  \arrow[r, "\cup"] &
    \quad \Ho^{r+s}(G_1,C_1) \arrow[d, "{(\phi,\psi_C)^*}"] \\
    \Ho^{r}(G_2,A_2) & \times & \Ho^{s}(G_2,B_2) \arrow[u, "{(\chi,\xi)^*}"] \quad \arrow[r, "\cup"] & \quad \Ho^{r+s}(G_2,C_2) .
\end{tikzcd}
\]
\end{prop}

\begin{proof}
(a)\quad Take $f\in \C^{r}(G_1,A_1)$ and $g\in \C^{s}(G_2,B_2)$. We have on the one hand
\[
(\phi,\psi_C)^*(f\cup (\chi,\xi)^*g)= (\phi,\psi_C)^*(f\cup (\xi\circ g\circ\chi))= \psi_{C}\circ (f\cup (\xi\circ g\circ\chi))\circ\phi,
\]
and on the other hand
\[
(\phi,\psi_A)^*(f)\cup g= 
(\psi_{A}\circ f\circ\phi)\cup g.
\]
We now show that both cochains are equal. For $(\sigma_1,\cdots,\sigma_{r+s}) \in G_2^{r+s}$ we have
\[
\begin{split}
(&\psi_{C}\circ (f\cup (\xi\circ g\circ\chi))\circ\phi)(\sigma_1,\cdots,\sigma_{r+s})\\
= 
&\psi_{C}\circ(f\cup (\xi\circ g\circ\chi)) (\phi\sigma_1,\cdots,\phi\sigma_{r+s})\\
=&\psi_{C}\circ (f(\phi\sigma_1,\cdots,\phi\sigma_{r})\times \phi\sigma_1\cdots\phi\sigma_{r}\cdot (\xi\circ g\circ\chi)(\phi\sigma_{r+1},\cdots,\phi\sigma_{r+s})).
\end{split}
\]
Since $(\chi,\xi)$ is a morphism of discrete modules, the latter expression is equal to
\[
\begin{split}
    &\psi_{C}\circ (f(\phi\sigma_1,\cdots,\phi\sigma_{r})\times
    \xi\circ(   \chi\phi\sigma_1\cdots\chi\phi\sigma_{r}\cdot g(\chi\phi\sigma_{r+1},\cdots,\chi\phi\sigma_{r+s}))=\\
    &\psi_{C}\circ (f(\phi\sigma_1,\cdots,\phi\sigma_{r})\times \xi\circ(   \sigma_1\cdots\sigma_{r}\cdot g(\sigma_{r+1},\cdots,\sigma_{r+s})),
\end{split}
\]
which by the commutativity of diagram (\ref{twisted preserving of bilinear maps}), is
\[
\begin{split}
&\psi_{A}\circ f(\phi\sigma_1,\cdots,\phi\sigma_{r}) \times \sigma_1\cdots\sigma_{r}\cdot  g(\sigma_{r+1},\cdots,\sigma_{r+s}) = \\
&(\psi_{A}\circ f\circ\phi)(\sigma_1,\cdots,\sigma_{r})\times \sigma_1\cdots\sigma_{r}\cdot  g(\sigma_{r+1},\cdots,\sigma_{r+s}) = \\
&((\psi_{A}\circ f\circ\phi)\cup g) (\sigma_1,\cdots,\sigma_{r+s}). 
\end{split}
\]
(b)\quad This follows immediately from (a) and the definition of cup products between cohomology classes (\ref{Cup product between cohomology classes}). 
\end{proof}

\begin{cor}
\label{twisted functoriality of cup product with a fixed group}
For a fixed profinite group $G$, and discrete $G$-module morphisms $\psi_A:A_1\to A_2, \psi_B:B_2\to B_1, \psi_C:C_1\to C_2$, if the diagram
\[
    \begin{tikzcd}
    A_1  \qquad
    \times \qquad 
    B_1 \quad \arrow[d, "\psi_A", xshift=-8ex] \arrow[r] &
    \quad C_1 \arrow[d, "\psi_C"] \\
    A_2 \qquad \times \qquad B_2 \quad  \arrow[u, "\psi_B", xshift=6ex] \arrow[r] & \quad C_2 
    \end{tikzcd}
\]
is commutative, then the diagrams
\[
\begin{tikzcd}
    \C^{r}(G,A_1) \arrow[d, "\psi_{A*}"] &
    \times &
    \C^{s}(G,B_1) \quad  \arrow[r, "\cup"] &
    \quad \C^{r+s}(G,C_1) \arrow[d, "\psi_{C*}"] \\
    \C^{r}(G,A_2) & \times & \C^{s}(G,B_2) \arrow[u, "\psi_{B*}"] \quad \arrow[r, "\cup"] & \quad \C^{r+s}(G,C_2) ,
\end{tikzcd}
\]
\[
\begin{tikzcd}
    \Ho^{r}(G,A_1) \arrow[d, "\psi_{A*}"] &
    \times &
    \Ho^{s}(G,B_1) \quad  \arrow[r, "\cup"] &
    \quad \Ho^{r+s}(G,C_1) \arrow[d, "\psi_{C*}"] \\
    \Ho^{r}(G,A_2) & \times & \Ho^{s}(G,B_2) \arrow[u, "\psi_{B*}"] \quad \arrow[r, "\cup"] & \quad \Ho^{r+s}(G,C_2) ,
\end{tikzcd}
\]
are commutative.

\end{cor}

\medskip
\medskip

\section{Definition of the augmented cup product}
\label{section about Augmented cup products}

\subsection{Tate Products}
\label{section about Tate products}
\quad\\
In order to define the augmented cup product we need to define one more notion, and that is the notion of a \textit{Tate product}.

\begin{defin}
\label{Tate product}
A \textbf{Tate product} is a data consisting of two short exact sequences of discrete $G$-modules:
\[
    0 \xrightarrow{} A'\xrightarrow{i} A\xrightarrow{j} A''\xrightarrow{} 0  , \]\[
    0 \xrightarrow{} B'\xrightarrow{i} B\xrightarrow{j} B''\xrightarrow{} 0  , 
\]
and two $G$-bilinear maps:
\[
    A'\times B\to C   \quad , \quad    A\times B'\to C ,
\]
coinciding on $A'\times B'$.  Such data is denoted by $(A,B,C)$.
\end{defin}

This definition can be generalized by considering $A,B,C$ as objects of any abelian category. For our purposes it is sufficient to restrict the definition to the category of $G$-modules as above. 
\begin{exam}
\label{Example of Tate product from one short exact sequence}
\rm
Let $A'\leq A$ and $C$ be discrete $G$-modules. If we have a $G$-bilinear map $\circ: A'\times A\to C$, which commutes on $A'\times A'$, then $(A,A,C)$ forms a Tate product. To see this, consider the short exact sequence
\[
0\to A'\to A\to A/A'\to 0,
\]
and the pairings
\[
\begin{split}
A'\times A\to C \qquad &, \qquad A\times A'\to C \\    
(a',a)\mapsto a'\circ a  \qquad &, \qquad (a,a')\mapsto a'\circ a.
\end{split}
\]
\end{exam}
\medskip
The Tate products form a category. Namely, consider Tate products $(A_1,B_1,C_1)$ over a profinite group $G_1$, and $(A_2,B_2,C_2)$ over a profinite group $G_2$. A \textbf{morphism of Tate products}, 
\[
(\phi,\psi):(A_1,B_1,C_1)\to(A_2,B_2,C_2),
\]
consists of a morphism of profinite groups $\phi: G_2\to G_1$, and morphisms of abelian groups 
\[
\psi_{A'}:A_1'\to A_2', \quad \psi_{A}:A_1\to A_2, \quad \psi_{A''}:A_1''\to A_2'',
\]
\[
\psi_{B'}:B_1'\to B_2', \quad \psi_{B}:B_1\to B_2, \quad \psi_{B''}:B_1''\to B_2'', 
\]
\[
\psi_{C}:C_1\to C_2,
\]
such that: \\
(1) \quad The diagrams \begin{equation}
\label{commutative diagram A of Tate product morphism}
    \begin{tikzcd}
0 \arrow[r] & 
A_1'\arrow[r, "i"]\arrow[d,"\psi_{A'}"] & 
A_1\arrow[r, "j"]\arrow[d,"\psi_A"] &
A_1''\arrow[r]\arrow[d,"\psi_{A''}"] &
0 \\ 
0 \arrow[r] &
A_2'\arrow[r, "i"] & 
A_2\arrow[r, "j"] & 
A_2''\arrow[r] &
0,
    \end{tikzcd}
\end{equation}

\begin{equation}
\label{commutative diagram B of Tate product morphism}
    \begin{tikzcd}
0 \arrow[r] & 
B_1'\arrow[r, "i"]\arrow[d,"\psi_{B'}"] & 
B_1\arrow[r, "j"]\arrow[d,"\psi_B"] &
B_1''\arrow[r]\arrow[d,"\psi_{B''}"] &
0 \\ 
0 \arrow[r] &
B_2'\arrow[r, "i"] & 
B_2\arrow[r, "j"] & 
B_2''\arrow[r] &
0,
    \end{tikzcd}
\end{equation}
are exact and commutative.\\
(2) \quad The morphisms 
\[
(\phi,\psi_{A'}),(\phi,\psi_A),(\phi,\psi_{A''}),(\phi,\psi_{B'}),(\phi,\psi_B),(\phi,\psi_{B''}),(\phi,\psi_C)
\] 
are morphisms of discrete modules. \\
(3) \quad The bilinear maps are preserved. Namely, the following diagrams are commutative:
\begin{center}
        \begin{tikzcd}
        A_1'\quad \arrow[d,"\psi_{A'}", xshift=-5ex] \times \quad
        B_1\arrow[d,"\psi_B", xshift=5ex]\arrow[r] &
        C_1\arrow[d,"\psi_C"]\\
        A_2' \quad \times
        \quad B_2\arrow[r] &
        C_2 , 
        \end{tikzcd}
        \quad
        \begin{tikzcd}
        A_1\quad \arrow[d,"\psi_A", xshift=-5ex] \times \quad
        B_1'\arrow[d,"\psi_{B'}", xshift=5ex]\arrow[r] &
        C_1\arrow[d,"\psi_C"]\\
        A_2 \quad \times
        \quad B_2'\arrow[r] &
        C_2 .
        \end{tikzcd}
\end{center}

\begin{rem}
\label{Remark about the existence of the cup products A'XB, AXB'}
\rm The $G$-bilinear maps $ A'\times B\to C$ and $A\times B'\to C$, give rise to the cup products:
\[
\begin{split}
    &\cup : {\rm C}^r(G,A')\times {\rm C}^s(G,B)\to {\rm C}^{r+s}(G,C) ,\\
    &\cup : {\rm C}^r(G,A)\times {\rm C}^s(G,B')\to {\rm C}^{r+s}(G,C) ,\\
    &\cup : \Ho^r(G,A')\times \Ho^s(G,B)\to \Ho^{r+s}(G,C) ,\\
    &\cup : \Ho^r(G,A)\times \Ho^s(G,B')\to \Ho^{r+s}(G,C) .
    \end{split}
\]
\end{rem}

\medskip
\medskip

\subsection{The augmented cup product}
\quad\\
For the rest of this paper, we fix a Tate product $(A,B,C)$ over $G$.

Recall that by the notations in Section \ref{Induced morphisms of cochains and cohomology groups}, the morphisms 
\begin{equation}
    \label{the morphisms j}
    j:A\to A'',\quad j:B\to B'',
\end{equation}
induce the morphisms of cochains
\begin{equation}
    \label{the morphisms j_*}
j_*: \C^{r}(G,A)\to \C^{r}(G,A''), \quad j_*: \C^{r}(G,B)\to \C^{r}(G,B'').
\end{equation}
Also recall that $A'$ and $B'$ may be identified with the kernels of the morphisms in (\ref{the morphisms j}), respectively, and that $\C^{r}(G,A')$ and $\C^{r}(G,B')$ may be identified with the kernels of the morphisms in (\ref{the morphisms j_*}), respectively.

\begin{defin}
\label{augmented cup product}
We define the \textbf{augmented cup product}
\[
\cup_{\aug}: {\rm H}^r(G,A'')\times {\rm H}^s(G,B'')\to {\rm H}^{r+s+1}(G,C)
\]
as follows:\\
Given $\alpha''\in {\rm H}^r(G,A{''})$, $\beta''\in {\rm H}^s(G,B'')$, we choose $f''\in {\rm Z}^r(G,A'')$, $g{''}\in {\rm Z}^s(G,B'')$ such that $\alpha''=[f'']$, $\beta''=[g'']$. We choose $f\in {\rm C}^r(G,A)$, $g\in {\rm C}^s(G,B)$ such that $j_*f=f''$, $j_*g=g''$, and define:
\[
\alpha''\cup_{\aug}\beta'':= [d_rf\cup g + (-1)^rf\cup d_sg].
\]
\end{defin}

\begin{prop}
\label{well definednece of augmented cup product}
The augmented cup product $\cup_{\aug}$ is well defined, i.e.:
(a) The cup products in the definition above have meanings.\\
(b) $d_rf\cup g + (-1)^rf\cup d_sg$ is a cocycle, so it indeed has a cohomology class.\\
(c) The construction is independent of the choice of the representatives.

\end{prop}

\begin{proof}
(a) By Proposition \ref{(a) df is in A' (b)f is in A' up to coboundary}(a), $d_rf\in {\rm C}^{r+1}(G,A')$ and $d_sg\in {\rm C}^{s+1}(G,B')$. By Remark \ref{Remark about the existence of the cup products A'XB, AXB'}, the cup products $d_rf\cup g, f\cup d_sg$ have meanings.\\
(b) By (\ref{Leibnitz of cup products}), 
\[
d_{r+s+1}(d_rf\cup g)=d_{r+1}(d_rf)\cup g + (-1)^{r+1}d_rf\cup d_sg=(-1)^{r+1}d_rf\cup d_sg,
\]
and
\[
d_{r+s+1}(f\cup d_sg)=d_rf\cup d_sg +(-1)^{r}f\cup d_{s+1}(d_sg)=d_rf\cup d_sg.
\]
Therefore we have:
\[
d_{r+s+1}(d_rf\cup g + (-1)^rf\cup d_sg)=(-1)^{r+1}d_rf\cup d_sg+(-1)^rd_rf\cup d_sg=0 .
\]
(c) Consider $f_1, f_2\in {\rm C}^r(G,A)$ with $j_*(f_1),j_*(f_2)\in \Z^r(G,A'')$ and $[j_*(f_1)]=[j_*(f_2)]=\alpha''$. Then $j_*(f_1-f_2)\in {\rm B}^r(G,A'')$. By Proposition \ref{(a) df is in A' (b)f is in A' up to coboundary}(b), $f_1-f_2$ is in ${\rm C}^r(G,A')$ up to a coboundary. Thus,
\[
f_1-f_2=f'+h,
\]
for some $f'\in {\rm C}^r(G,A')$ and $h\in {\rm B}^r(G,A)$.\\
Hence, 
\[
\begin{split}
&[d_rf_1\cup g + (-1)^rf_1\cup d_sg]-[d_rf_2\cup g + (-1)^rf_2\cup d_sg] \\
=&[d_r(f_1-f_2)\cup g + (-1)^r(f_1-f_2)\cup d_sg] \\
=&[d_r(f'+h)\cup g + (-1)^r(f'+h)\cup d_sg] \\
=&[d_rf'\cup g + (-1)^rf'\cup d_sg]+[d_rh\cup g + (-1)^rh\cup d_sg].    
\end{split}
\]
The separation in the last row to two different cohomology classes is possible, since the corresponding cochains are indeed cocycles (and even coboundaries) as we now show. 
In view of Remark \ref{Remark about the existence of the cup products A'XB, AXB'}, there is a cup product $f'\cup g$, and 
(\ref{Leibnitz of cup products}) gives,
\[
[d_rf'\cup g + (-1)^rf'\cup d_sg]=0.
\] 
Further, $d_rh=0$, and by Proposition \ref{Cup product between cocycle and coboundary}, $h\cup d_sg$ is a coboundary (as a cup product of two coboundaries). Hence, 
\[
[d_rh\cup g + (-1)^rh\cup d_sg]=0.
\]
A similar argument shows that the construction is independent of the choice of different representatives $g_1, g_2\in {\rm C}^s(G,B)$.\qedhere  

\end{proof}

\begin{rem}
\rm If there exists a $G$-bilinear map $A\times B\to C$ which induces the maps $A'\times B\to C$ and $A\times B'\to C$, then the augmented cup product is $0$. Indeed, in this case, for $f,g$ as in Definition \ref{augmented cup product} we have a cup product $f\cup g$. Hence, by (\ref{Leibnitz of cup products}), 
\[
d_{r+s}(f\cup g)=d_rf\cup g + (-1)^rf\cup d_sg.
\]
Thus,
\[
\alpha''\cup_{\aug}\beta''= [d_rf\cup g + (-1)^rf\cup d_sg]=[d_{r+s}(f\cup g)]=0.
\]
\end{rem}

\begin{exam}
\label{acup 0X0->1}
\rm
We compute explicitly
\[
\cup_{\aug}: \Ho^0(G,A'')\times \Ho^0(G,B'')\to {\rm H}^{1}(G,C).
\]
Take $\alpha''\in \Ho^0(G,A'')=(A'')^G$ with representative $a\in \C^0(G,A)=A$, such that $\alpha''=[j_*a]$. Similarly, take $\beta''\in \Ho^0(G,B'')$ with representative $b\in B$, such that $\beta''=[j_*b]$. We show that:
\[
\alpha''\acup\beta''=[h],
\]
where $h\in \Z^1(G,C)$ is defined for $\sigma\in G$ by:
\[
h(\sigma)=(\sigma a-a)\times \sigma b+a\times (\sigma b-b).
\]
The multiplications on the right are those of the bilinear maps $A'\times B\to C$, $A\times B'\to C$. 

By definition,
\[
\alpha''\acup\beta''=[d_0a\cup b + a\cup d_0b],
\]
where the $1$-cocycle in the RHS is:
\[
\begin{split}
    (d_0a\cup b + a\cup d_0b)(\sigma)&=(d_0a\cup b)(\sigma) + (a\cup d_0b)(\sigma)=\\
    d_0a(\sigma)\times\sigma b+a\times d_0b(\sigma)&=(\sigma a-a)\times \sigma b+a\times (\sigma b-b).
\end{split}
\]

Note that since 
\[j(\sigma a-a)=j(\sigma a)-ja=\sigma (ja)-ja=ja-ja=0,
\]
$(\sigma a-a)$ is an element of $A'=\Ker j$. Similarly, $(\sigma b-b)\in B'$. 
\end{exam}

\medskip

For the rest of this paper, when we write that a cochain $f$ is a \textbf{representative} of a cohomology class $\alpha''$, we mean that it satisfies
\[
\alpha''=[j_*f].
\]
\medskip

\section{Properties of the augmented cup product}
\label{Section about properties of the augmented cup product}

\subsection{Graded commutativity}
\begin{prop}
For every $\alpha''\in\Ho^r(G,A''), \beta''\in\Ho^s(G,B'')$, we have
\[
\alpha''\acup\beta'' = (-1)^{rs}\cdot \beta''\acup\alpha''.
\]
\end{prop}
\begin{proof}
    Let $f,g$ be representatives of $\alpha'',\beta''$ respectively, as in the definition of the augmented cup product. Then:
    \[
    \alpha''\cup_{\aug}\beta''= [d_rf\cup g + (-1)^rf\cup d_sg],
    \]
    \[
    \beta''\cup_{\aug}\alpha''= [d_sg\cup f+ (-1)^{s}g\cup d_rf].
    \]
    By the graded commutativity of the cup product \cite{Koch02}*{Theorem 3.27},
    \[
    \begin{split}
    \alpha''\cup_{\aug}\beta''&= [(-1)^{(r+1)s}g\cup d_rf + (-1)^{r(s+1)}(-1)^rd_sg\cup f]\\
    &= (-1)^{rs}[(-1)^{s}g\cup d_rf + (-1)^{2r}d_sg\cup f]\\
    &= (-1)^{rs}[d_sg\cup f+ (-1)^{s}g\cup d_rf]=(-1)^{rs}\cdot \beta''\acup\alpha''. \qedhere    
    \end{split}
    \]
\end{proof}

\medskip
\medskip

\subsection{Compatibility with the $G$-action}
\quad\\
For a discrete $G$-module $A$, define an action of $G$ on $\C^r(G,A)$ by:
\[
(\sigma\cdot f)(\tau_1,\ldots,\tau_r)=\sigma f(\sigma^{-1}\tau_1,\ldots,\sigma^{-1}\tau_r).
\] 
This action is compatible with the coboundary maps, hence it induces an action on $\Ho^r(G,A)$:
\[
\sigma\cdot [f]=[\sigma\cdot f].
\]
Further, the action commutes with morphisms of discrete modules, and hence with morphisms of cochains. Namely, consider a morphism of profinite groups $\phi: G_2\to G_1$ and a morphism $\psi: A_1\to A_2$ such that $(\phi,\psi)$ is a morphism of discrete modules. For the induced morphism of cochains 
\[
(\phi,\psi)^*: \C^r(G_1,A_1)\to \C^r(G_2,A_2) ,
\]
the diagram
\begin{center}
    \begin{tikzcd}
    \C^r(G_1,A_1)\arrow[r,"{(\phi,\psi)^*}"]\arrow[d,"\sigma"]&
    \C^r(G_2,A_2)\arrow[d,"\sigma"] \\
    \C^{r}(G_1,A_1)\arrow[r,"{(\phi,\psi)^*}"] &
    \C^{r}(G_2,A_2)
    \end{tikzcd}
\end{center}
is commutative.
In particular, recall that the Tate product $(A,B,C)$ consists of a $G$-module morphism $j:A\to A''$. For the induced morphism of cochains $j_*:\C^r(G,A)\to\C^r(G,A'')$ we have:
\[
\sigma\cdot(j_*f)=j_*(\sigma\cdot f).
\]

The following proposition shows that the $G$-action is compatible with the augmented cup product.
\begin{prop}
\label{Acup is G-bilinear}
For every $\sigma\in G$, $\alpha''\in {\rm H}^r(G,A{''})$, $\beta''\in {\rm H}^s(G,B'')$ we have:
\[
\sigma \cdot (\alpha'' \cup_{\aug} \beta'') = \sigma \cdot \alpha'' \cup_{\aug} \sigma \cdot \beta''.
\]
\end{prop}

\begin{proof}
Let $f,g$ be representatives of $\alpha'',\beta''$, respectively. Since the cup product is $G$-bilinear, and the $G$-action is compatible with the coboundary maps, 
\begin{equation}
    \label{G-action on acup}
\begin{split}
\sigma \cdot (\alpha'' \cup_{\aug} \beta'')&=\sigma\cdot [d_rf\cup g + (-1)^rf\cup d_sg]\\
&=[\sigma\cdot (d_{r}f)\cup \sigma \cdot g+ (-1)^{r}\sigma\cdot f\cup \sigma \cdot (d_{s}g)]\\
&=[d_{r}(\sigma\cdot f)\cup \sigma \cdot g + (-1)^{r}\sigma\cdot f\cup d_{s}(\sigma \cdot g)].
\end{split}
\end{equation}
Since
\[
\sigma\cdot\alpha''=\sigma\cdot [j_*f]=[\sigma\cdot (j_*f)]=[j_*(\sigma\cdot f)],
\]
$\sigma\cdot f$ is a representative of $\sigma\cdot\alpha''$. Similarly, $\sigma\cdot g$ is a representative of $\sigma\cdot\beta''$.
Hence, the final expression in (\ref{G-action on acup}) is
\[ 
\sigma \cdot \alpha'' \cup_{\aug} \sigma \cdot \beta''. \qedhere
\]
\end{proof}

\medskip
\medskip

\subsection{Functoriality}
\quad\\
By the discussion in Section \ref{section about Tate products}, given a morphism of Tate products $(\phi,\psi):(A_1,B_1,C_1)\to(A_2,B_2,C_2)$, we have induced morphisms of cochains and cohomology groups for any two corresponding modules of the Tate products. Namely,
\[
\begin{split}
&(\phi,\psi_{A'})^*: \C^r(G_1,A_1')\to \C^r(G_2,A_2'), \quad \Ho^r(G_1,A_1')\to \Ho^r(G_2,A_2'),\\
&(\phi,\psi_{A})^*: \C^r(G_1,A_1)\to \C^r(G_2,A_2), \quad \Ho^r(G_1,A_1)\to \Ho^r(G_2,A_2),\\
&(\phi,\psi_{A''})^*: \C^r(G_1,A_1'')\to \C^r(G_2,A_2''), \quad \Ho^r(G_1,A_1'')\to \Ho^r(G_2,A_2''),\\
&(\phi,\psi_{B'})^*: \C^s(G_1,B_1')\to \C^s(G_2,B_2'), \quad \Ho^s(G_1,B_1')\to \Ho^s(G_2,B_1'),\\
&(\phi,\psi_{B})^*: \C^s(G_1,B_1)\to \C^s(G_2,B_2), \quad \Ho^s(G_1,B_1)\to \Ho^s(G_2,B_2),\\
&(\phi,\psi_{B''})^*: \C^s(G_1,B_1'')\to \C^s(G_2,B_2''), \quad \Ho^s(G_1,B_1'')\to \Ho^s(G_2,B_2''),\\
&(\phi,\psi_{C})^*: \C^{r+s}(G_1,C_1)\to \C^{r+s}(G_2,C_2), \quad \Ho^{r+s}(G_1,C_1)\to \Ho^{r+s}(G_2,C_2).
\end{split}
\]
Furthermore, by Proposition \ref{functoriality of the cup product} and condition (3) in the definition of a morphism of Tate products in Section \ref{section about Tate products}, the diagrams
\begin{equation}
    \label{compatibility 1 of Tate product morphism with the cup product}
\begin{tikzcd}
    \C^{r}(G_1,A_1') \arrow[d, "{(\phi,\psi_{A'})^*}"] &
    \times &
    \C^{s}(G_1,B_1) \quad \arrow[d, "{(\phi,\psi_B)^*}"] \arrow[r, "\cup"] &
    \quad \C^{r+s}(G_1,C_1) \arrow[d, "{(\phi,\psi_C)^*}"] \\
    \C^{r}(G_2,A_2') & \times & \C^{s}(G_2,B_2) \quad \arrow[r, "\cup"] & \quad \C^{r+s}(G_2,C_2) ,
\end{tikzcd}
\end{equation}
\begin{equation}
    \label{compatibility 2 of Tate product morphism with the cup product}
\begin{tikzcd}
    \C^{r}(G_1,A_1) \arrow[d, "{(\phi,\psi_{A})^*}"] &
    \times &
    \C^{s}(G_1,B_1') \quad \arrow[d, "{(\phi,\psi_{B'})^*}"] \arrow[r, "\cup"] &
    \quad \C^{r+s}(G_1,C_1) \arrow[d, "{(\phi,\psi_C)^*}"] \\
    \C^{r}(G_2,A_2) & \times & \C^{s}(G_2,B_2') \quad \arrow[r, "\cup"] & \quad \C^{r+s}(G_2,C_2) ,
\end{tikzcd}
\end{equation}
are commutative.

The following proposition shows the functoriality of the augmented cup product.

\begin{prop}
\label{functoriality of Acup}
For a morphism of Tate products $(\phi,\psi):(A_1,B_1,C_1)\to(A_2,B_2,C_2)$, the diagram
\[
\begin{tikzcd}
    \Ho^{r}(G_1,A_1'') \arrow[d, "{(\phi,\psi_{A''})^*}"] &
    \times &
    \Ho^{s}(G_1,B_1'') \quad \arrow[d, "{(\phi,\psi_{B''})^*}"] \arrow[r, "\acup"] &
    \quad \Ho^{r+s+1}(G_1,C_1) \arrow[d, "{(\phi,\psi_C)^*}"] \\
    \Ho^{r}(G_2,A_2'') & \times & \Ho^{s}(G_2,B_2'') \quad \arrow[r, "\acup"] & \quad \Ho^{r+s+1}(G_2,C_2) 
\end{tikzcd}
\]
is commutative.
\end{prop}

\begin{proof}
Take $\alpha''\in \Ho^{r}(G_1,A_1'')$ and $\beta''\in \Ho^{s}(G_1,B_1'')$ with representatives $f\in \C^r(G_1,A_1)$ and $g\in \C^s(G_1,B_1)$, respectively (namely, $\alpha''=[j_*f], \beta''=[j_*g]$). 

Note that
\[
(\phi,\psi_{A''})^*(\alpha'')=(\phi,\psi_{A''})^*([j_*f])= [\psi_{A''}\circ j\circ f \circ \phi].
\]
By the commutativity of diagram (\ref{commutative diagram A of Tate product morphism}), the latter expression is equal to
\[
[j\circ \psi_{A}\circ f \circ \phi]= [j_*(\phi,\psi_{A})^*(f)].
\]
Thus, $(\phi,\psi_{A})^*(f)$ is a representative of $(\phi,\psi_{A''})^*(\alpha'')$. Similarly, $(\phi,\psi_{B})^*(g)$ is a representative of $(\phi,\psi_{B''})^*(\beta'')$.

By the definition of the augmented cup product,
\[
(\phi,\psi_{C})^*(\alpha''\cup_{\aug} \beta'')=
(\phi,\psi_{C})^*[d_rf\cup g+(-1)^rf\cup d_sg].
\]
By the commutativity of diagrams (\ref{compatibility 1 of Tate product morphism with the cup product}) and (\ref{compatibility 2 of Tate product morphism with the cup product}), the latter expression is equal to
\[
[(\phi,\psi_{A'})^*(d_rf) \cup (\phi,\psi_{B})^*(g) +(-1)^r(\phi,\psi_{A})^*(f) \cup (\phi,\psi_{B'})^*(d_sg)],
\]
which, by the commutativity of diagram (\ref{compatibility of induced morphism with the coboundary maps}), is equal to
\[
\begin{split}
    &[d_r((\phi,\psi_{A})^*(f)) \cup (\phi,\psi_{B})^*(g) +(-1)^r(\phi,\psi_{A})^*(f)\cup d_s((\phi,\psi_{B})^*(g))]=\\
    &(\phi,\psi_{A''})^*\alpha''\cup_{\aug} (\phi,\psi_{B''})^* \beta''. \qedhere
 \end{split} 
\] 
\end{proof}

Suppose that we have the following setting: A fixed profinite group $G$, Tate products $(A_1,B_1,C_1), (A_2,B_2,C_2)$ over $G$, and discrete $G$-module morphisms 
\begin{equation}
\label{commutative diagram A of twisted Tate product morphism}
    \begin{tikzcd}
0 \arrow[r] & 
A_1'\arrow[r, "i"]\arrow[d,"\psi_{A'}"] & 
A_1\arrow[r, "j"]\arrow[d,"\psi_A"] &
A_1''\arrow[r]\arrow[d,"\psi_{A''}"] &
0 \\ 
0 \arrow[r] &
A_2'\arrow[r, "i"] & 
A_2\arrow[r, "j"] & 
A_2''\arrow[r] &
0,
    \end{tikzcd}
\end{equation}

\begin{equation}
\label{commutative diagram B of twisted Tate product morphism}
    \begin{tikzcd}
0 \arrow[r] & 
B_1'\arrow[r, "i"] & 
B_1\arrow[r, "j"] &
B_1''\arrow[r] &
0 \\ 
0 \arrow[r] &
B_2'\arrow[r, "i"] \arrow[u,"\psi_{B'}"] & 
B_2\arrow[r, "j"] \arrow[u,"\psi_B"] & 
B_2''\arrow[r] \arrow[u,"\psi_{B''}"] &
0,
    \end{tikzcd}
\end{equation}
\[
    C_1\xrightarrow{\psi_C} C_2,
\]
such that diagrams (\ref{commutative diagram A of twisted Tate product morphism}) and (\ref{commutative diagram B of twisted Tate product morphism}) are exact and commutative.

The following proposition shows the functoriality of the augmented cup product in this setting.

\begin{prop}
    \label{twisted functoriality of acup}
    For the setting as above, if the diagrams
    \[
    \begin{tikzcd}
        A_1'\quad \arrow[d,"\psi_{A'}", xshift=-5ex] \times \quad
        B_1\arrow[r] &
        C_1\arrow[d,"\psi_C"]\\
        A_2' \quad \times
        \quad B_2\arrow[r] \arrow[u,"\psi_B", xshift=5ex] &
        C_2 , 
        \end{tikzcd}
        \quad
        \begin{tikzcd}
        A_1\quad \arrow[d,"\psi_A", xshift=-5ex] \times \quad
        B_1'\arrow[r] &
        C_1\arrow[d,"\psi_C"]\\
        A_2 \quad \times
        \quad B_2' \arrow[u,"\psi_{B'}", xshift=5ex] \arrow[r] &
        C_2 
        \end{tikzcd}
    \]
    are commutative, then the diagram
    \[
    \begin{tikzcd}
    \Ho^{r}(G,A_1'') \arrow[d, "\psi_{A''*}"] &
    \times &
    \Ho^{s}(G,B_1'') \quad \arrow[r, "\acup"] &
    \quad \Ho^{r+s+1}(G,C_1) \arrow[d, "\psi_{C*}"] \\
    \Ho^{r}(G,A_2'') & \times & \Ho^{s}(G,B_2'')  \quad \arrow[u, "\psi_{B''*}"] \arrow[r, "\acup"] & \quad \Ho^{r+s+1}(G,C_2) 
\end{tikzcd}
    \]
    is commutative.
\end{prop}

\begin{proof}
By Corollary \ref{twisted functoriality of cup product with a fixed group}, the diagrams
\begin{equation}
    \label{commutative diagram 1 of cup products in Tate product}
    \begin{tikzcd}
        \C^{r}(G,A_1')\quad \arrow[d,"\psi_{A'*}", xshift=-9ex] \times \quad
        \C^{s}(G,B_1)\arrow[r, "\cup"] &
        \C^{r+s}(G,C_1)\arrow[d,"\psi_{C*}"]\\
        \C^{r}(G,A_2') \quad \times
        \quad \C^{s}(G,B_2)\arrow[r, "\cup"] \arrow[u,"\psi_{B*}", xshift=9ex] &
        \C^{r+s}(G,C_2) , 
        \end{tikzcd}
\end{equation}
\begin{equation}
    \label{commutative diagram 2 of cup products in Tate product}
        \begin{tikzcd}
        \C^{r}(G,A_1)\quad \arrow[d,"\psi_{A*}", xshift=-9ex] \times \quad
        \C^{s}(G,B_1')\arrow[r, "\cup"] &
        \C^{r+s}(G,C_1) \arrow[d,"\psi_{C*}"]\\
        \C^{r}(G,A_2) \quad \times
        \quad \C^{s}(G,B_2') \arrow[u,"\psi_{B'*}", xshift=9ex] \arrow[r, "\cup"] &
        \C^{r+s}(G,C_2) ,
        \end{tikzcd}
\end{equation}
are commutative.

Take $\alpha''\in \Ho^{r}(G,A_1'')$ and $\beta''\in \Ho^{s}(G,B_2'')$ with representatives $f\in \C^r(G,A_1)$ and $g\in \C^s(G,B_2)$, respectively (namely, $\alpha''=[j_*f], \beta''=[j_*g]$). 
Note that
\[
\psi_{A''*}(\alpha'')=\psi_{A''*}([j_*f])= [\psi_{A''}\circ j\circ f].
\]
By the commutativity of diagram (\ref{commutative diagram A of Tate product morphism}), the latter expression is equal to
\[
[j\circ \psi_{A}\circ f]= [j_*\psi_{A*}(f)].
\]
Thus, $\psi_{A*}(f)$ is a representative of $\psi_{A''*}(\alpha'')$. Similarly, $\psi_{B*}(g)$ is a representative of $\psi_{B''*}(\beta'')$.

By the definition of the augmented cup product,
\[
\begin{split}
    \psi_{C*}(\alpha''\acup \psi_{B''*}(\beta''))&=
    \psi_{C*}[d_rf\cup \psi_{B*}(g)+(-1)^rf\cup d_s(\psi_{B*}(g))]\\
    &= \psi_{C*}[d_rf\cup \psi_{B*}(g)+(-1)^rf\cup \psi_{B'*}(d_sg)].
\end{split}
\]
By the commutativity of diagrams (\ref{commutative diagram 1 of cup products in Tate product}) and (\ref{commutative diagram 2 of cup products in Tate product}), the latter expression is equal to
\[
\begin{split}
&[\psi_{A'*}(d_rf) \cup g +(-1)^r\psi_{A*}(f) \cup d_sg]=\\
&[d_r(\psi_{A*}(f)) \cup g +(-1)^r\psi_{A*}(f) \cup d_sg]= \psi_{A''*}(\alpha'')\acup \beta''. \qedhere
\end{split}
\]
\end{proof}

\medskip
\medskip

\subsection{Compatibility with the connecting homomorphism}
\quad\\
Recall that a short exact sequence of discrete $G$-modules  
\[
0 \xrightarrow{} A'\xrightarrow{i} A\xrightarrow{j} A''\xrightarrow{} 0 ,
\]
induces a long exact sequence of cohomology groups
\[
\cdots \xrightarrow{} \Ho^r(G,A') \xrightarrow{} \Ho^r(G,A)\xrightarrow{} \Ho^r(G,A'')\xrightarrow{\delta} \Ho^{r+1}(G,A')\xrightarrow{} \Ho^{r+1}(G,A)\xrightarrow{} \cdots \]
The connecting homomorphism $\delta$ is defined as follows: Given $\alpha''\in \Ho^r(G,A'')$, take $f''\in \Z^r(G,A'')$ such that $\alpha''=[f'']$. Choose $f\in \C^r(G,A)$ such that $j_*f=f''$, and set
\[
\delta(\alpha'')=[d_rf]\in \Ho^{r+1}(G,A').
\]
\cite{NSW}*{Theorem 1.3.2} shows that $\delta$ is well defined.

\begin{prop}
\label{compatibility with connecting homomorphism}

In the following diagram:
\begin{center}
        \begin{tikzcd}
        \Ho^r(G,A)\arrow[d, "j_*"] &
        \times &
        \Ho^{s+1}(G,B') \quad\quad \arrow[r,"\cup"] &
        \quad\quad\Ho^{r+s+1}(G,C)\arrow[d,equal]\\
        \Ho^r(G,A'')\arrow[d,"\delta"] &
        \times &
        \Ho^s(G,B'') \quad\quad \arrow[u,"\delta"]\arrow[r,"\acup"] &
        \quad\quad\Ho^{r+s+1}(G,C)\arrow[d,equal]\\
        \Ho^{r+1}(G,A') &
        \times &
        \Ho^s(G,B) \quad\quad \arrow[u, "j_*"]\arrow[r,"\cup"] &
        \quad\quad\Ho^{r+s+1}(G,C) ,
        \end{tikzcd}
\end{center}
    (1) \quad The upper square is commutative of character $(-1)^r$.\\
    (2) \quad The lower square is commutative. 

\end{prop}

\begin{proof}
    (1) \quad Take $\alpha\in\Ho^r(G,A), \beta''\in \Ho^s(G,B'')$.  We show that 
    \[
    j_*(\alpha)\acup\beta''=(-1)^r\alpha\cup \delta(\beta'').
    \]
    Take $f\in Z^r(G,A)$ such that $\alpha=[f]$, and let $g\in \C^s(G,B)$ be a representative with $[j_*g]=\beta''$. Since $d_rf=0$, we have:
    \[
    \begin{split}
    j_*(\alpha)\acup\beta''&= [j_*f]\acup[j_*g]=[d_rf\cup g + (-1)^rf\cup d_sg]\\
    &=[(-1)^rf\cup d_sg]=(-1)^r[f]\cup [d_sg] =(-1)^r\alpha\cup \delta(\beta'').     
    \end{split}
    \]
    \medskip
    (2) \quad Take $\alpha''\in\Ho^r(G,A''), \beta\in \Ho^s(G,B)$. We show that 
    \[
    \alpha''\acup j_*(\beta)=\delta(\alpha'')\cup \beta.
    \]
    Take $g\in Z^s(G,B)$ such that $\beta=[g]$, and let $f\in \C^r(G,A)$ be a representative with $[j_*f]=\alpha''$. Since $d_sg=0$, we have:
    \[
    \begin{split}
         \alpha''\acup j_*(\beta)&= [j_*f]\acup[j_*g]=[d_rf\cup g + (-1)^rf\cup d_sg]\\
         &=[d_rf\cup g]=[d_rf]\cup [g]= \delta(\alpha'')\cup \beta. \qedhere
    \end{split}
    \]
\end{proof}

\medskip
\medskip

\subsection{Induced pairings}
\quad\\
Suppose that the short exact sequence $0\xrightarrow{} B'\xrightarrow{} B\xrightarrow{} B''\xrightarrow{}0$, from the Tate product $(A,B,C)$, can be extended to an exact and commutative diagram of discrete $G$-modules
\begin{equation}
    \label{extendable}
        \begin{tikzcd}
        &
        0 \arrow[d] & \\
        &
        B' \arrow[d]  &  \\
        0 \arrow[r] &
        B \arrow[r] \arrow[d] &
        D \arrow[r] \arrow[d] &
        E \arrow[r] \arrow[d, equal] &
        0 \\
        0 \arrow[r] &
        B'' \arrow[r, "\phi"] \arrow[d] &
        D'' \arrow[r, "\psi"] \arrow[d] &
        E \arrow[r] &
        0 \\
        &
        0 &
        0 & & .
        \end{tikzcd}
\end{equation}
Then we have, for $s\geq 1$, the exact sequences
\begin{equation}
    \label{long exact sequence from ''}
    \Ho^{s-1}(E)\xrightarrow{\delta} \Ho^{s}(B'')\xrightarrow{\phi_*}\Ho^{s}(D'')\xrightarrow{\psi_*}\Ho^{s}(E) ,
\end{equation}
\begin{equation}
    \label{long exact sequence from null}
    \Ho^{s-1}(E)\xrightarrow{\delta} \Ho^{s}(B)\xrightarrow{}\Ho^{s}(D)\xrightarrow{}\Ho^{s}(E) ,
\end{equation}
\begin{equation}
    \label{long exact sequence from B}
    \Ho^{s-1}(B'')\xrightarrow{\delta} \Ho^{s}(B')\xrightarrow{}\Ho^{s}(B)\xrightarrow{}\Ho^{s}(B'') 
\end{equation} 
(in these statements we omit the references to $G$). We use the compatibility of the augmented cup product with the connecting homomorphism (Proposition \ref{compatibility with connecting homomorphism}), to show that in this case the augmented cup product induces an additional pairing. This will be useful for the geometrical applications is Section \ref{Applications}.

\begin{thm}
\label{How to induce rho}
Suppose that we have a Tate product $(A,B,C)$ and a commutative diagram as in (\ref{extendable}). Then: \\
(1)  \quad  The restricted pairing
 \[
\acup: \Img (\Ho^r(A)\xrightarrow{j_*} \Ho^r(A'')) \times \Ker(\Ho^{s}(B'')\xrightarrow{\phi_*}\Ho^{s}(D'')) \to \Ho^{r+s+1}(C)
\]
is trivial. \\ 
(2) \quad The augmented cup product $\acup$ induces a bilinear map:
\[
\Img (\Ho^r(A)\xrightarrow{j_*} \Ho^r(A'')) \times \Img(\Ho^{s}(B'')\xrightarrow{\phi_*}\Ho^{s}(D'')) \to \Ho^{r+s+1}(C).
\]
\end{thm}

\begin{proof}
(1) \quad Take $\alpha''\in \Img(j_*)\subseteq \Ho^{r}(A'')$ and $\alpha\in \Ho^r(A)$ such that $j_*(\alpha)=\alpha''$. Take $\beta''\in \Ker(\phi_*)\subseteq \Ho^{s}(B'')$. By the exactness of $(\ref{long exact sequence from ''})$, there exists $\gamma\in\Ho^{s-1}(E)$ such that $\delta(\gamma)=\beta''$. By Proposition \ref{compatibility with connecting homomorphism} (1), 
\begin{equation}
\label{compute by compatibility of connecting hom.}
\alpha''\acup\beta''=(-1)^r\alpha\cup\delta(\beta'')=(-1)^r\alpha\cup\delta(\delta(\gamma)).    
\end{equation}
By diagram (\ref{extendable}) and the functoriality of the connecting homomorphism \cite{NSW}*{Proposition 1.3.3}, we have the commutative diagram with an exact column:
\begin{center}
    \begin{tikzcd} 
    \Ho^{s-1}(E) \arrow[d, equal] \arrow[r, "\delta"] &
    \Ho^{s}(B) \arrow[d] \\
    \Ho^{s-1}(E) \arrow[r, "\delta"] &
    \Ho^{s}(B'') \arrow[d, "\delta"] \\
    & \Ho^{s+1}(B').  
    \end{tikzcd}
\end{center}
Hence, the composition
\[
\Ho^{s-1}(E)\xrightarrow{\delta} \Ho^{s}(B'')\xrightarrow{\delta} \Ho^{s+1}(B')
\]
is zero.
Specifically,  $\delta(\delta(\gamma))=0$, and we get (\ref{compute by compatibility of connecting hom.})  
\[
\alpha''\acup\beta''=0.
\]
(2) \quad This follows from (1) and the bilinearity of the augmented cup product.
\end{proof}

\begin{cor}
\label{corollary with E=Z}
In the above setting, in the case that $E=\mathbb{Z}$ (with trivial $G$-action), the augmented cup product $\acup$ induces a bilinear map:
\[
\Img (\Ho^r(A)\xrightarrow{j_*} \Ho^r(A'')) \times \Ho^{1}(D'') \to \Ho^{r+2}(C).
\]
\end{cor}

\begin{proof}
The group $\Ho^{1}(\mathbb{Z})$ consists of the continuous homomorphisms from $G$ to $\mathbb{Z}$. Since $G$ is compact and $\mathbb{Z}$ is discrete and torsion-free, there are no such homomorphisms except the trivial homomorphism. Hence, 
$\Ho^{1}(\mathbb{Z})=0$. Thus, by the exactness of $(\ref{long exact sequence from ''})$, $\Ho^{1}(B'')\xrightarrow{\phi_*}\Ho^{1}(D'')$ is surjective, and by Theorem \ref{How to induce rho} (2), we get the requested bilinear map.    
\end{proof}

\medskip
\medskip

\subsection{Compatibility with the restriction and the corestriction maps}
Using the discussion in \cite{Serre}*{Chapter I, \S 2.5}, we give the following interpretation of the restriction and the corestriction maps.

Let $H$ be a closed subgroup of $G$, and let $C$ be a discrete $H$-module. The (co-)induced module $\Ind(C)$ is the group of continuous maps $c^*$ from $G$ to $C$ such that $c^*(hx)=h\cdot c^*(x)$ for every $h\in H$, and $x\in G$. The $G$-module structure of $\Ind(C)$ is given for $c^*\in \Ind(C)$ and $g,x\in G$ by
\[
(gc^*)(x)=c^*(xg).
\]
The functor $\Ind$ is an exact functor from the category of $H$-modules to the category of $G$-modules \cite{Serre}*{Chapter I, \S 2.5}.

The discrete $H$-module morphism
\[
\begin{split}
    e = e_C: \Ind(C)&\to C \\
    c^*&\mapsto c^*(1),
\end{split}
\]
together with the inclusion map $\incl: H\hookrightarrow G$, induce the cohomology group homomorphism
\[
e^*=(\incl,e)^*: \Ho^{r}(G,\Ind(C))\to \Ho^{r}(H,C).
\]
By Shapiro's lemma \cite{Serre}*{Ch. I, \S 2.5, Proposition 10}, $e^*$ is an isomorphism.

Given a bilinear map of discrete $H$-modules $C_1\times C_2 \to C_3$, we have a natural bilinear map of discrete $G$-modules
    \[
    \Ind(C_1)\times \Ind(C_2) \to \Ind(C_3),
    \]
defined by
\[
(c_1^*\times c_2^*)(x)=c_1^*(x)\times c_2^*(x).
\]

\begin{lem}
\label{e preserves bilinear maps}
The following diagram commutes:
\[
\begin{tikzcd} 
    \Ind(C_1)  \quad \times \quad \Ind(C_2) \arrow[d, "e_{C_1}", xshift=-10ex] \arrow[d, "e_{C_2}", xshift=10ex]  \arrow[r] &
    \Ind(C_3) \arrow[d, "e_{C_3}"] \\
    C_1 \qquad \quad
    \times \quad \qquad C_2 \arrow[r] &
     C_3 .
    \end{tikzcd}
\]
\end{lem}

\begin{proof}
Take $c_1^*\in \Ind(C_1)$ and $c_2^*\in \Ind(C_2)$. We have
\[
e_{C_3}(c_1^*\times c_2^*) = (c_1^*\times c_2^*)(1) = c_1^*(1) \times c_1^*(1) = e_{C_1}(c_1^*) \times e_{C_2}(c_2^*). \qedhere
\]
\end{proof}

There is a discrete $H$-module morphism
\[
\begin{split}
    i: C&\to \Ind(C) \\
    c&\mapsto i(c): G\to C\\
    &\qquad \qquad x\mapsto x\cdot c.
\end{split}
\]
It induces the cohomology group homomorphism
\[
i_*: \Ho^{r}(G,C)\to \Ho^{r}(G,\Ind(C)).
\]
The map $i_*$ coincides with the restriction map $\res: \Ho^{r}(G,C)\to \Ho^{r}(H,C)$, in the sense that the diagram
\begin{equation}
\label{i is res}
    \begin{tikzcd} 
    \Ho^{r}(G,C) \arrow[r, "i_*"'] \arrow[rr, "\res", bend left=15] &
    \Ho^{r}(G,\Ind(C)) \arrow[r, "\sim", "e^*"'] &
    \Ho^{r}(H,C)
    \end{tikzcd}
\end{equation}
is commutative.

If $H$ is open in $G$ (hence of finite index $n=[G:H]$), then there is a discrete $H$-module morphism
\[
\begin{split}
    \pi: \Ind(C)&\to C \\
    c^*&\mapsto \sum_{x\in G/H}x\cdot c^*(x^{-1}),
\end{split}
\]
where $x$ varies through representatives of the cosets of $H$ in $G$. It induces the cohomology group homomorphism
\[
\pi_*: \Ho^{r}(G,\Ind(C))\to \Ho^{r}(G,C).
\]
The map $\pi_*$ coincides with the corestriction map $\cores: \Ho^{r}(H,C)\to \Ho^{r}(G,C)$, in the sense that the diagram
\begin{equation}
\label{pi is cores}
    \begin{tikzcd} 
    \Ho^{r}(H,C) \arrow[r, "\sim", "(e^*)^{-1}"'] \arrow[rr, "\cores", bend left=15] &
    \Ho^{r}(G,\Ind(C)) \arrow[r, "\pi_*"'] &
    \Ho^{r}(G,C)
    \end{tikzcd}
\end{equation}
is commutative.

\begin{lem}
\label{pi almost preserves bilinear maps}
 Given a bilinear map of discrete $H$-modules $C_1\times C_2 \to C_3$, the following diagram commutes:
\[
\begin{tikzcd} 
    \Ind(C_1)  \quad \times \quad \Ind(C_2) \arrow[d, "\pi_{C_1}", xshift=-10ex] \arrow[r] &
    \Ind(C_3) \arrow[d, "\pi_{C_3}"] \\
    C_1 \qquad \quad
    \times \quad \qquad C_2 \arrow[u, "i_{C_2}", xshift=10ex] \arrow[r] &
     C_3 .
    \end{tikzcd}
\]
\end{lem}

\begin{proof}
Take $c_1^*\in \Ind(C_1)$ and $c_2\in C_2$. We have
\[
\pi_{C_3}(c_1^*\times i_{c_2}(c_2)) = \sum_{x\in G/H}x\cdot(c_1^*\times i_{c_2}(c_2))(x^{-1})
= \sum_{x\in G/H}(x\cdot c_1^*(x^{-1})\times x\cdot i_{c_2}(c_2)(x^{-1}))
\]
\[
= \sum_{x\in G/H}(x\cdot c_1^*(x^{-1})\times x\cdot x^{-1}\cdot c_2) 
= \sum_{x\in G/H}x\cdot c_1^*(x^{-1})\times c_2= \pi_{C_1}(c_1^*)\times c_2. \qedhere
\]
\end{proof}

Consider a Tate product of discrete $H$-modules $(A,B,C)$. Since $\Ind$ is an exact functor, we have short exact sequences 
    \[
    \begin{split}
        0\xrightarrow{}\Ind(A')\xrightarrow{}\Ind(A)\xrightarrow{}\Ind(A'')\xrightarrow{} 0 ,\\
        0\xrightarrow{}\Ind(B')\xrightarrow{}\Ind(B)\xrightarrow{}\Ind(B'')\xrightarrow{} 0.
    \end{split}
    \]
Further, we have natural pairings
    \[
    \Ind(A')\times \Ind(B) \to \Ind(C), \qquad
    \Ind(A)\times \Ind(B') \to \Ind(C),
    \]
which coincide on $\Ind(A')\times \Ind(B')$, as described above. Thus, we get an induced Tate product of discrete $G$-modules $(\Ind(A),\Ind(B),\Ind(C))$, which gives rise to an augmented cup product
\[
\acup: \Ho^{r}(G,\Ind(A''))\times \Ho^{s}(G,\Ind(B''))\to \Ho^{r+s+1}(G,\Ind(C)).
\]

\begin{thm}
For an open subgroup $H\leq G$, the diagram
\begin{center}
        \begin{tikzcd}
        \Ho^r(H,A'')\arrow[d,"\cores"] &
        \times &
        \Ho^s(H,B'') \quad\quad \arrow[r,"\acup"] &
        \quad\quad\Ho^{r+s+1}(H,C)\arrow[d,"\cores"]\\
        \Ho^r(G,A'') &
        \times &
        \Ho^s(G,B'') \arrow[u,"\res"] \quad\quad \arrow[r,"\acup"] &
        \quad\quad\Ho^{r+s+1}(G,C) 
        \end{tikzcd}
\end{center}
is commutative.
\end{thm}

\begin{proof}
By (\ref{i is res}) and (\ref{pi is cores}), the desired commutativity is equivalent to the commutativity of 
\begin{equation}
\label{equivalent diagram}
    \begin{tikzcd} 
    \Ho^r(H,A'')\arrow[d,"(e_A^*)^{-1}", "\wr"'] &
    \times &
    \Ho^s(H,B'') \quad\quad \arrow[r,"\acup"] &
    \quad\quad\Ho^{r+s+1}(H,C)\arrow[d,"(e_C^*)^{-1}", "\wr"']\\
    \Ho^r(G,\Ind(A'')) \arrow[d, "\pi_{A*}"] &
    \times &
    \Ho^s(G,\Ind(B'')) \arrow[u,"e_B^*", "\wr"'] \quad\quad \arrow[r,"\acup"] &
    \quad\quad\Ho^{r+s+1}(G,\Ind(C)) \arrow[d, "\pi_{C*}"] \\
    \Ho^r(G,A'') &
    \times &
    \Ho^s(G,B'') \arrow[u,"i_{B*}"] \quad\quad \arrow[r,"\acup"] &
    \quad\quad\Ho^{r+s+1}(G,C).
    \end{tikzcd}
\end{equation}

By Lemma \ref{e preserves bilinear maps}, the map $e^*$ preserves bilinear maps. Hence, by the functoriality of the augmented cup product (Proposition \ref{functoriality of Acup}), the upper square of diagram (\ref{equivalent diagram}) is commutative.

By Lemma \ref{pi almost preserves bilinear maps}, the assumptions for Proposition \ref{twisted functoriality of acup} are satisfied, which implies that the lower square of diagram (\ref{equivalent diagram}) is commutative, and hence the whole diagram is commutative.
\end{proof}

\medskip
\medskip

\section{Applications}
\label{Applications}
In the important paper of Lichtenbaum from 1969 \cite{Lich}, he constructs three compatible pairings, involving the Brauer and Picard groups of curves and their cohomology groups, and shows that over $p$-adic fields they are perfect (see the introduction). Two of those pairings are special cases of the augmented cup product. Lichtenbaum just mentions this in a brief sentence, with no further information. However, his construction is direct, in the level of cochains, without any use of the general machinery of the augmented cup products as given here in Chapters \ref{section about Augmented cup products} and \ref{Section about properties of the augmented cup product}. Our aim is to give an alternative approach, and to construct those pairings using the general theory of the augmented cup products.
\\

\subsection{Geometrical background}
\label{section about geometrical bacground}
\quad\\
Let us introduce briefly the geometrical notions in Lichtenbaum's construction. Our principal reference for basic notations of curves over fields is \cite{Silverman}*{Chapter 2}. One can find further relevant information in  \cite{Jarden05}*{Chapter 5},  \cite{Jarden11}*{Chapter 11} or \cite{Efrat01}*{Section 1}.

Let $k$ be a field, $\bar{k}$ the separable closure of $k$, and $G=\Gal(\bar{k}/k)$ the absolute Galois group of $k$. Let $X$ be a smooth, proper, geometrically connected curve over $k$, and let $K$ be the function field of $X$ (see \cite{Silverman}*{Chapter 2} for detailed definitions).
Let $\bar{X}$ be the extension of $X$ by scalars to $\bar{k}$ (i.e. $\bar{X}=X\otimes_{k}\bar{k}$), and $\bar{K}$ the function field of $\bar{X}$. The \textbf{divisor group} of a curve $X$, denoted by $\Div(X)$, is the free abelian group on the points of $X$. Namely, its elements are the formal sums $D=\sum_{P\in X}{n_{P}P}$, where $n_P\in\mathbb{Z}$, and $n_P=0$ for all but finitely many points. The \textbf{support} $|D|$ of the divisor $D$ is $\{P\in X| \quad n_{P}\neq 0\}$. We have the \textbf{degree homomorphism}
\[
\begin{split}
\Deg: \Div(X)&\to \mathbb{Z}\\
D \quad &\mapsto \sum_{P\in X}{n_P[k(P):k]}.
\end{split}
\]
Here, $k(P)$ is the residue field of $P$. Denote by $\Div_0(X)$ the kernel of $\Deg$. 

For a curve $\bar{X}$ over $\bar{k}$, there exists $P\in \bar{X}$ with $\bar{k}(P)=\bar{k}$ \cite{Liu}*{Chapter 3, Proposition 2.20}. Hence the degree homomorphism $\Deg: \Div(\bar{X})\to \mathbb{Z}$ is surjective, and we have a short exact sequence
\begin{equation}
    \label{short exact sequence with deg}
    0\xrightarrow{} \Div_0(\bar{X})\xhookrightarrow{} \Div(\bar{X}) \xrightarrow{\Deg} \mathbb{Z} \xrightarrow{} 0.
\end{equation}

Back to the case of a curve $X$ over a general field $k$ (not necessarily separably closed). Consider the \textbf{divisor map}
\[
\begin{split}
    \fdiv: K^*&\to \Div(X)\\
    f &\mapsto \sum_{P\in X}{v_P(f)P}.
\end{split}
\]
Here, $v_P$ is a valuation on $K$ corresponding to the point $P$. It is a non-trivial discrete valuation which is trivial on $k$. 

A divisor $D$ is called a \textbf{principal divisor} if $D\in\Img(\fdiv)$. The \textbf{Picard group} $\Pic(X)$ of $X$ is the cokernel of the map $\fdiv$. We have
\[
\begin{split}
&(1)\quad \Ker(\fdiv)=k^* \hbox{ and} \\
&(2)\quad \Deg\circ\fdiv=0
\end{split}
\]\cite{Silverman}*{Chapter 2, Proposition 3.1}. As a consequence, the group $\Img(\fdiv)$ of principal divisors is actually a subgroup of $\Div_0(X)$. Let $\Pic_0(X)=\Div_0(X)/\Img(\fdiv)$ be the quotient group. Further, we get a well-defined induced degree homomorphism
\[
\begin{split}
\Deg: \Pic(X)&\to \mathbb{Z}\\
E \quad &\mapsto \Deg(D),
\end{split}
\]
where $D$ is a divisor in the class of $E$, and a well-defined induced divisor map
\[
\begin{split}
\fdiv: K^*/k^* &\to \Div_0(X)\\
\tilde{f} \quad &\mapsto \fdiv(f),
\end{split}
\]
where $f\in K^*$ is in the class of $\tilde{f}$.
Thus we have an exact and commutative diagram:
\begin{equation}
\label{diagram of 2 short exact sequences from div to pic}
    \begin{tikzcd}
    & & 0 \arrow[d] & 0 \arrow[d] & \\
    0 \arrow[r] &
    K^*/k^* \arrow[r,"\fdiv"] \arrow[d,equal] &
    \Div_0(X) \arrow[r] \arrow[d,hook] &
    \Pic_0(X) \arrow[r] \arrow[d,hook] &
    0 \\
    0 \arrow[r] &
    K^*/k^* \arrow[r,"\fdiv"] &
    \Div(X) \arrow[r] \arrow[d,"\Deg"] &
    \Pic(X) \arrow[r] \arrow[d,"\Deg"] &
    0 \\
    & & \mathbb{Z} \arrow[r, equal] & \mathbb{Z} & .
    \end{tikzcd}    
\end{equation}

By (\ref{short exact sequence with deg}), we have a short exact sequence
\[
0\xrightarrow{} \Pic_0(\bar{X}) \xrightarrow{} \Pic(\bar{X}) \xrightarrow{\Deg} \mathbb{Z} \xrightarrow{} 0 ,
\]
which induces the exact sequence of cohomology groups \begin{equation}
\label{long exact sequence with pic}
 \Ho^1(G,\Pic_0(\bar{X})) \xrightarrow{} \Ho^1(G,\Pic(\bar{X})) \xrightarrow{} \Ho^1(G,\mathbb{Z}).    
\end{equation}
Recall that $\Ho^1(G,\mathbb{Z})=0$ (see the proof of Corollary \ref{corollary with E=Z}). Thus the map $\Ho^1(G,\Pic_0(\bar{X}))\to \Ho^1(G,\Pic(\bar{X}))$ is surjective.

The \textbf{evaluation} of a function $f\in K^*$ at a point $P\in X$ is the result of the substitution of the coordinates of $P$ in the variables of $f$ (recall that $f$ is simply a quotient of two polynomials with coefficients in $k$). Hence we have $f(P)\in k \cup \{ \infty\}$.
The evaluation of a function $f\in K^*$ at a divisor $D=\sum_{P\in X}{n_{P}P}$ with $f(P)\notin \{0,\infty\}$ for every $P$ in the support $|D|$ of $D$ is defined by
\[
f(D)= \prod_{P\in |D|}f(P)^{n_{P}}\in k^*.
\]
By Weil reciprocity law (\cite{Lang83}*{Page 172}, \cite{Silverman}*{Exercise 2.11}), for $f,g\in \bar{K}^*$ with $|\fdiv(f)|\cap |\fdiv(g)|=\emptyset$ we have
\[
f(\fdiv(g))=g(\fdiv(f)).
\]

The group $G$ acts on $P=(x_1,\ldots,x_n)\in \bar{X}$ by the Galois action on each coordinate, i.e.
\[
\sigma\cdot P = (\sigma(x_1),\ldots,\sigma(x_n)),
\]
and acts on $D=\sum_{P\in \bar{X}}n_{P}P\in \Div(\bar{X})$ by
\[
\sigma\cdot D=\sum_{P\in \bar{X}}n_{P}(\sigma\cdot P).
\]
This makes $\Div(\bar{X})$ a $G$-module.

The Brauer group $\Br(X)$ of a curve $X$ is defined in \cite{Lich}*{Section 1} as the kernel of the induced homomorphism 
\[
\fdiv_*:\Ho^2(G,\bar{K}^*)\to \Ho^2(G,\Div(\bar{X})).
\]
In \cite{Lich}*{Appendix} it is shown that this definition of the Brauer group is equivalent to the usual one using \'etale cohomology.

By diagram (\ref{diagram of 2 short exact sequences from div to pic}) and the functoriality of the connecting homomorphism, we have an exact and commutative diagram:
\begin{equation}
    \label{long exact diagram from div to pic}
    \begin{tikzcd}
    \Ho^{1}(G,\Div_0(\bar{X})) \arrow[d] \arrow[r] &
    \Ho^{1}(G,\Pic_0(\bar{X})) \arrow[d] \arrow[r, "\delta"] &
    \Ho^{2}(G,\bar{K}^*/\bar{k}^*) \arrow[d, equal] \arrow[r] &
    \Ho^{2}(G,\Div_0(\bar{X})) \arrow[d] \\
    \Ho^{1}(G,\Div(\bar{X})) \arrow[r] &
    \Ho^{1}(G,\Pic(\bar{X})) \arrow[r, "\delta"] &
    \Ho^{2}(G,\bar{K}^*/\bar{k}^*) \arrow[r] &
    \Ho^{2}(G,\Div(\bar{X})) .
    \end{tikzcd}
\end{equation}
As pointed out in \cite{Lich}*{Section 2}, $\Ho^{1}(G,\Div(\bar{X}))=0$, hence the following diagram is exact and commutative:
\begin{equation}
    \label{The basic exact diagram}
    \begin{tikzcd} [column sep=small]
     0 \arrow[d] &
    0 \arrow[d]  \\
    \Br(X) \arrow[d, hook] &
    \Ho^{1}(G,\Pic(\bar{X})) \arrow[d,"\delta"]  \\
    \Ho^{2}(G,\bar{K}^*) \arrow[d] \arrow[r] &
    \Ho^{2}(G,\bar{K}^*/\bar{k}^*) \arrow[d] \\
    \Ho^{2}(G,\Div(\bar{X})) \arrow[r, equal] & \Ho^{2}(G,\Div(\bar{X})). 
    \end{tikzcd}
\end{equation}
Let $\phi: \Br(X)\to \Ho^{2}(G,\bar{K}^*/\bar{k}^*)$ be the composition of the embedding $\Br(X)\hookrightarrow \Ho^{2}(G,\bar{K}^*)$ and the induced homomorphism  $\Ho^{2}(G,\bar{K}^*)\to \Ho^{2}(G,\bar{K}^*/\bar{k}^*)$.
By diagram (\ref{The basic exact diagram}), for every element in the image of $\phi$ there exists a unique preimage in $\Ho^{1}(G,\Pic(\bar{X}))$, and we get an induced homomorphism $h: \Br(X)\to \Ho^1(G,\Pic(\bar{X}))$, making the left triangle in the following diagram commutative:
\begin{equation}
    \label{factoring j, daigram}
    \begin{tikzcd} [column sep=small]
    \Br(X) \arrow[dr, "\phi"'] \arrow[r, "h"] &
    \Ho^{1}(G,\Pic(\bar{X})) \arrow[d,"\delta"] &
    \Ho^{1}(G,\Pic_0(\bar{X})) \arrow[d,"\delta"] \arrow[l, two heads]   \\
    & \Ho^{2}(G,\bar{K}^*/\bar{k}^*) \arrow[r, equal] &
    \Ho^{2}(G,\bar{K}^*/\bar{k}^*).
    \end{tikzcd}
\end{equation}
By diagram (\ref{long exact diagram from div to pic}), the right square of diagram (\ref{factoring j, daigram}) is also commutative.
\begin{lem}
\label{The image of phi is in the image of delta}
The image of $\phi: \Br(X)\xrightarrow{} \Ho^{2}(G,\bar{K}^*/\bar{k}^*)$ is contained in the image of $\delta: \Ho^{1}(G,\Pic_0(\bar{X}))\xrightarrow{} \Ho^{2}(G,\bar{K}^*/\bar{k}^*)$.
\end{lem}
\begin{proof}
It follows from diagram (\ref{factoring j, daigram}) and the surjectivity of the map
\[
\Ho^{1}(G,\Pic_0(\bar{X}))\to \Ho^{1}(G,\Pic(\bar{X})) \hbox{ (diagram (\ref{long exact sequence with pic})).} \qedhere
\]
\end{proof}

\medskip
\medskip

\subsection{The pairing $\rho_0$ - Lichtenbaum's construction} 
\quad\\
Consider the following setting. Let $X$ be a proper, smooth, geometrically connected curve over a field $k$, with separable closure $\bar{k}$ and absolute Galois group $G$. Let $\bar{X}$ be the extension of $X$ by scalars to $\bar{k}$. Recall that the Brauer group $\Br(k)$ of $k$ is canonically isomorphic to $\Ho^2(G,\bar{k}^*)$.

Lichtenbaum's construction of the pairing
\[
     \rho_0: \Ho^0(G,\Pic_0(\bar{X}))\times \Ho^1(G,\Pic_0(\bar{X}))\to\Br(k)
\]
is as follows: Take $x\in \Ho^0(G,\Pic_0(\bar{X}))$ and $\alpha\in \Ho^1(G,\Pic_0(\bar{X}))$. Choose $y\in \Z^0(G,\Pic_0(\bar{X}))$ representing $x$, and choose $E\in \C^0(G,\Div_0(\bar{X}))=\Div_0(\bar{X})$ mapping onto $y$. Choose $a_{\sigma}\in \Z^1(G,\Pic_0(\bar{X}))$ representing $\alpha$, and choose $b_{\sigma}\in \C^1(G,\Div_0(\bar{X}))$ mapping onto $a_{\sigma}$. The coboundaries of these representatives, namely $d_0E\in \C^1(G,\Div_0(\bar{X}))$ and $d_1b_{\sigma}\in \C^2(G,\Div_0(\bar{X}))$, are in the image of the induced morphism $\fdiv_*$. Namely, there exist $f_{\sigma,\tau}\in \C^{2}(G,\bar{K}^*/\bar{k}^*)$ and  $g_{\sigma}\in\C^{1}(G,\bar{K}^*/\bar{k}^*)$ such that $\fdiv_*(g_{\sigma})=d_0E$ and $\fdiv_*(f_{\sigma,\tau})=d_1b_{\sigma}$. It follows from Proposition \ref{(a) df is in A' (b)f is in A' up to coboundary}(a). Now define,
\[
\rho_0(x,\alpha)=[g_{\sigma}(\sigma b_{\tau}) \cdot f_{\sigma,\tau}(E)].
\]
Lichtenbaum shows directly that this is independent of the various choices made and that $g_{\sigma}(\sigma b_{\tau}) \cdot f_{\sigma,\tau}(E)$ is indeed an element of $\Z^2(G,\bar{k}^*)$, so it has a cohomology class which is in $\Ho^2(G,\bar{k}^*) = \Br(k)$.\\

\subsection{The pairing $\rho_0$ - interpretation in terms of the augmented cup product}
\quad\\
Let $\Bar{K}$ be the field of rational functions on $\Bar{X}$. Consider the short exact sequence of $G$-modules,
\[
0\xrightarrow{}\bar{K}^*/\bar{k}^*\xrightarrow{\fdiv}\Div_0(\bar{X})\xrightarrow{j}\Pic_0(\bar{X})\xrightarrow{}0.
\]
We have a natural pairing:
\begin{equation}
\label{The pairing which give rise to the augmented cup product}
\begin{split}
\bar{K}^*/\bar{k}^*\times \Div_0(\bar{X})&\to \bar{k}^*  \\ (f,D) &\mapsto f(D).
\end{split}    
\end{equation}
Due to Weil reciprocity law, this pairing commutes on $\bar{K}^*/\bar{k}^* \times \bar{K}^*/\bar{k}^*$. Thus, in view of Example \ref{Example of Tate product from one short exact sequence}, we get a Tate product, and can define for every non-negative integers $r,s$ an augmented cup product:
\[
\acup: \Ho^r(G,\Pic_0(\bar{X})) \times \Ho^s(G,\Pic_0(\bar{X})) \to \Ho^{r+s+1}(G,\bar{k}^*).
\]
In particular, we get an augmented cup product:
\[
\acup: \Ho^0(G,\Pic_0(\bar{X})) \times \Ho^1(G,\Pic_0(\bar{X})) \to \Ho^2(G,\bar{k}^*)= \Br(k).
\]

We now show that Lichtenbaum's $\rho_0$ is exactly this augmented cup product. Consider $x\in \Ho^0(G,\Pic_0(\bar{X}))$ and $\alpha\in \Ho^1(G,\Pic_0(\bar{X}))$ as before. In order to compute $x\acup \alpha$, we should choose representatives $E\in\C^0(G,\Div_0(\bar{X}))$ such that $x=[j_*E]$, and $b\in\C^1(G,\Div_0(\bar{X}))$ such that $\alpha=[j_*b]$. Note that these are the same representatives from Lichtenbaum's construction. As shown in Proposition \ref{(a) df is in A' (b)f is in A' up to coboundary}(a), the coboundaries of $b$ and $E$ lie in the kernel of $j_*$. Use the following commutative diagram with exact rows to see this:
\begin{center}
        \begin{tikzcd}
        0\arrow[r] & {\rm C}^{0}(G,\bar{K}^*/\bar{k}^*)\arrow[r,"\fdiv_*"]\arrow[d,"d_0"] & {\rm C}^{0}(G,\Div_0(\bar{X}))\arrow[r,"j_*"]\arrow[d,"d_{0}"] & {\rm C}^{0}(G,\Pic_0(\bar{X}))\arrow[r]\arrow[d,"d_0"] & 0\\
        0\arrow[r] & {\rm C}^{1}(G,\bar{K}^*/\bar{k}^*)\arrow[r,"\fdiv_*"]\arrow[d,"d_1"] & {\rm C}^{1}(G,\Div_0(\bar{X}))\arrow[r,"j_*"]\arrow[d,"d_1"] & {\rm C}^{1}(G,\Pic_0(\bar{X}))\arrow[r]\arrow[d,"d_1"] & 0\\
        0\arrow[r] & {\rm C}^{2}(G,\bar{K}^*/\bar{k}^*)\arrow[r,"\fdiv_*"] & {\rm C}^{2}(G,\Div_0(\bar{X}))\arrow[r,"j_*"] & {\rm C}^{2}(G,\Pic_0(\bar{X}))\arrow[r] & 0.
        \end{tikzcd}
\end{center}
Hence, we can take $f\in{\C}^{2}(G,\bar{K}^*/\bar{k}^*)$ with $\fdiv_*(f)=d_1b$, and $g\in{\C}^{1}(G,\bar{K}^*/\bar{k}^*)$ with $\fdiv_*(g)=d_0E$. Since $\fdiv_*$ is an embedding, we may identify $f$ with $d_1b$, and $g$ with $d_0E$.

By the definition of the augmented cup product, we have:
\[
x\acup\alpha=[d_0E\cup b + (-1)^{0}\cdot E\cup d_1b]=[g\cup b + E\cup f].
\]
Computing the cup products by its definition, and passing from additive to multiplicative writing (for the multiplicative group $\bar{k}^*$), we get:
\begin{equation}
\label{rho_0 is acup}
\begin{split}
x\acup\alpha &= [(g\cup b)(\sigma,\tau) \cdot (E\cup f)(\sigma,\tau)]\\
&= [(g_{\sigma}\times\sigma b_{\tau})\cdot (E\times f_{\sigma,\tau})] = [g_{\sigma}(\sigma b_{\tau})\cdot f_{\sigma,\tau}(E)]\\
&=\rho_0(x,\alpha).
\end{split}
\end{equation}
Thus, $\rho_0$ is indeed an example of an augmented cup product, and it is well defined by the general theory (Proposition \ref{well definednece of augmented cup product}). 

\medskip
\medskip

\subsection{The pairing $\rho$ - Lichtenbaum's construction}
\quad\\
We now study the pairing
\[
\rho: \Pic_0(X) \times \Ho^1(G,\Pic(\bar{X})) \to \Br(k).
\]
Lichtenbaum defines this pairing as follows. He considers the short exact sequence of discrete $G$-modules 
\[
0\xrightarrow{} \Pic_0(\bar{X}) \xrightarrow{} \Pic(\bar{X}) \xrightarrow{\Deg} \mathbb{Z} \xrightarrow{} 0 ,
\]
and the exact sequence of cohomology groups it induces:
\begin{equation}
\label{phi is surjective}
\mathbb{Z}=\Ho^0(G,\mathbb{Z}) \xrightarrow{\delta} \Ho^1(G,\Pic_0(\bar{X})) \xrightarrow{\phi} \Ho^1(G,\Pic(\bar{X})) \xrightarrow{} \Ho^1(G,\mathbb{Z}) =0.    
\end{equation}
Since $\Pic_0(X)\subseteq \Ho^0(G,\Pic_0(\bar{X}))$ and $\phi$ is surjective, by proving that the restricted pairing 
\[
     \rho_0: \Pic_0(X)\times \Ker(\phi)\to\Br(k)
\]
is trivial, he gets that $\rho_0$ induces a pairing $\rho$ as above. Lichtenbaum's proof of the triviality of the restricted pairing is by computations based on the geometrical properties of the divisors groups (see \cite{Lich}*{Setion 4} for details).

\medskip
\medskip

\subsection{The pairing $\rho$ - interpretation in terms of the augmented cup product}
\quad\\
We now show that this construction is just a specific example of a more general result from the theory of the augmented cup product that we have developed in Chapter \ref{Section about properties of the augmented cup product}. The essence of the generalization is to prove that the restricted pairing is trivial, using general cohomological arguments as in Theorem \ref{How to induce rho} (1). Recall that $\rho_0$ is an augmented cup product, and that the short exact sequence of the Tate product which gives rise to it is:
\[
0\xrightarrow{}\bar{K}^*/\bar{k}^*\xrightarrow{}\Div_0(\bar{X})\xrightarrow{}\Pic_0(\bar{X})\xrightarrow{}0.
\]
This short exact sequence can be extended to an exact and commutative diagram
\begin{equation}
\label{The extended diagram}
        \begin{tikzcd}
        &
        0 \arrow[d] & 0 \arrow[d] & \\
        &
        \bar{K}^*/\bar{k}^* \arrow[d] \arrow[r, equal]  &
        \bar{K}^*/\bar{k}^* \arrow[d]  &\\
        0 \arrow[r] &
        \Div_0(\bar{X}) \arrow[r] \arrow[d] &
        \Div(\bar{X}) \arrow[r, "\Deg"] \arrow[d] &
        \mathbb{Z} \arrow[r] \arrow[d, equal] &
        0 \\
        0 \arrow[r] &
        \Pic_0(\bar{X}) \arrow[r] \arrow[d] &
        \Pic(\bar{X}) \arrow[r, "\Deg"] \arrow[d] &
        \mathbb{Z} \arrow[r] &
        0 \\
        &
        0 &
        0 .
        \end{tikzcd}
\end{equation}
Hence, the assumptions of Corollary \ref{corollary with E=Z} are satisfied, and we get that the augmented cup product
\[
\acup: \Ho^0(G,\Pic_0(\bar{X})) \times \Ho^1(G,\Pic_0(\bar{X})) \to \Ho^2(G,\bar{k}^*)= \Br(k),
\]
induces a bilinear map
\[
\Img (\Ho^{0}(G,\Div_0(\bar{X}))\xrightarrow{} \Ho^{0}(G,\Pic_0(\bar{X}))) \times \Ho^{1}(G,\Pic(\bar{X})) \to \Br(k).
\]
Since $\Ho^{0}(G,\Div_0(\bar{X}))=\Div_0(X)$, the image of $\Ho^{0}(G,\Div_0(\bar{X}))\xrightarrow{} \Ho^{0}(G,\Pic_0(\bar{X}))$ is $\Pic_0(X)$. Thus, we get the requested pairing:
\[
\rho: \Pic_0(X) \times \Ho^1(G,\Pic(\bar{X})) \to \Br(k).
\]

Note that this is the same induced pairing as in Lichtenbaum's paper. In both cases, to compute $\rho(x,\alpha)$, for $x\in \Pic_0(X)$ and $\alpha \in \Ho^1(G,\Pic(\bar{X}))$, we take $\beta \in \Ho^1(G,\Pic_0(\bar{X}))$ mapping onto $\alpha$, and define 
\begin{equation}
\label{definition of rho}
\rho (x,\alpha)=\rho_0(x,\beta).
\end{equation}
\\
\subsection{The pairing $\psi$} 
\quad\\
This pairing is not a special case of the augmented cup product as the former two, but we show its compatibility with the previous pairings using the general theory from Chapter \ref{Section about properties of the augmented cup product}.
The pairing is 
\[
\begin{split}
     \psi: \Div(X)\times\Br(X)&\to\Br(k)\\
     (D,\alpha)&\mapsto [f_{\sigma,\tau}(D)],
\end{split}
\]
where $f_{\sigma,\tau}\in \Z^{2}(G,\bar{K}^*)$ is a $2$-cocycle representing $\alpha$. Lichtenbaum shows that $\psi$ is well-defined \cite{Lich}*{Section 3}, and that $\psi$ vanishes when $D$ is a principal divisor \cite{Lich}*{Theorem 1}. Hence, $\psi$ induces a pairing:
\begin{equation}
\label{definition of psi}
    \begin{split}
     \psi: \Pic(X)\times\Br(X)&\to\Br(k)\\
     (E,\alpha)&\mapsto [f_{\sigma,\tau}(D)].
\end{split}
\end{equation}
Here $D$ is a divisor mapping onto $E$, and $f_{\sigma,\tau}$ is as above.

\medskip
\medskip

\subsection{Compatibility of $\psi$ with $\rho$}
\quad\\
Lichtenbaum proves \cite{Lich}*{Section 4} that $\psi$ and $\rho$ are compatible, in the sense of the following proposition. We prove this compatibility using the general theory of the augmented cup products.  

\begin{prop}
\label{prop of compatibility of rho and psi}
Let $h:\Br(X)\to \Ho^1(G,\Pic(\bar{X}))$ be the homomorphism from diagram (\ref{factoring j, daigram}).  
The following diagram is commutative:
\begin{equation}
\label{compatibility of rho and psi}
\begin{tikzcd}
\rho: \quad \Pic_0(X) \quad  
\times \quad
\Ho^{1}(G,\Pic(\bar{X})) \arrow[d,hook, xshift=-9ex] \arrow[r] &
\Br(k) \arrow[d, equal] \\
\psi: \quad \Pic(X) \qquad
\times \qquad
\Br(X) \qquad \arrow[u, "h", xshift=9ex] \arrow[r] &
\Br(k).
\end{tikzcd}
\end{equation}
\end{prop}

\begin{proof}
First note that the pairing (\ref{The pairing which give rise to the augmented cup product}), which gives rise to our augmented cup product $\rho_0$, induces a cup product
\begin{equation}
    \label{the induced cup product}
    \begin{split}
    \cup: \Div_0(X)=\Ho^{0}(G,\Div_0(\bar{X})) &\times \Ho^{2}(G,\bar{K}^*/\bar{k}^*) \to \Ho^{2}(G,\bar{k}^*)=\Br(k) \\
    (D,[f_{\sigma,\tau}]) &\mapsto [f_{\sigma,\tau}(D)].
\end{split}
\end{equation}
By the definitions of the pairings (\ref{definition of psi}) and (\ref{the induced cup product}), the following diagram is commutative:
\begin{equation}
\label{compatibility of cup and psi}
\begin{tikzcd}
\Div_0(X) \quad  
\times \quad
\Ho^{2}(G,\bar{K}^*/\bar{k}^*) \arrow[d, xshift=-10ex] \arrow[r, "\cup"] &
\Br(k) \arrow[d, equal] \\
\Pic(X) \qquad
\times \qquad
\Br(X) \quad \arrow[u, "\phi", xshift=8ex] \arrow[r, "\psi"] &
\Br(k).
\end{tikzcd}
\end{equation}
Further, by (\ref{definition of rho}) and (\ref{rho_0 is acup}), the diagram
\begin{equation}
\label{compatibility of rho with acup}
\begin{tikzcd}
\Pic_0(X) \qquad \quad
\times \quad
\Ho^{1}(G,\Pic(\bar{X})) \arrow[d,hook, xshift=-12ex] \arrow[r, "\rho"] &
\Br(k) \arrow[d, equal] \\
\Ho^{0}(G,\Pic_0(\bar{X})) \quad
\times \quad
\Ho^{1}(G,\Pic_0(\bar{X})) \arrow[u, xshift=10ex] \arrow[r, "\rho_0", "\acup"'] &
\Br(k)  
\end{tikzcd}
\end{equation}
is commutative. Finally, by Proposition \ref{compatibility with connecting homomorphism} (1), the diagram
\begin{equation}
\label{compatibility of acup with cup}
\begin{tikzcd}
\Ho^{0}(G,\Pic_0(\bar{X})) \quad
\times \quad
\Ho^{1}(G,\Pic_0(\bar{X})) \arrow[d, xshift=11ex, "\delta"] \arrow[r, "\acup"] &
\Br(k) \arrow[d, equal] \\
\quad \Div_0(X) \qquad  
\times \quad
\Ho^{2}(G,\bar{K}^*/\bar{k}^*) \arrow[u, xshift=-11ex] \arrow[r, "\cup"] &
\Br(k)
\end{tikzcd}
\end{equation}
is also commutative.
Combining (\ref{compatibility of cup and psi}), (\ref{compatibility of rho with acup}) and (\ref{compatibility of acup with cup}), we get a commutative diagram:
\begin{equation}
\label{4 rows compatibility}
\begin{tikzcd}
\quad \Pic_0(X) \qquad 
\times \quad
\Ho^{1}(G,\Pic(\bar{X})) \arrow[d,hook, xshift=-11ex] \arrow[r, "\rho"] &
\Br(k) \arrow[d, equal] \\
\Ho^{0}(G,\Pic_0(\bar{X})) \quad
\times \quad
\Ho^{1}(G,\Pic_0(\bar{X})) \arrow[u, xshift=11ex] \arrow[d, xshift=11ex, "\delta"] \arrow[r, "\acup"] &
\Br(k) \arrow[d, equal] \\
\quad \Div_0(X) \qquad  
\times \quad
\Ho^{2}(G,\bar{K}^*/\bar{k}^*) \arrow[u, xshift=-11ex] \arrow[d, xshift=-11ex] \arrow[r, "\cup"] &
\Br(k) \arrow[d, equal] \\
\Pic(X) \qquad
\times \quad \qquad
\Br(X)  \arrow[u, "\phi"', xshift=11ex] \arrow[r, "\psi"] &
\Br(k).
\end{tikzcd}
\end{equation}
We call a 4-tuple of elements in
\[
\Pic_0(X) \times \Ho^0(G,\Pic_0(\bar{X})) \times \Div_0(X) \times \Pic(X)
\]
\textit{compatible} if its elements map one to each other under the relevant maps in the left column of (\ref{4 rows compatibility}). Note that every element of $\Pic_0(X)$ can be completed to a compatible 4-tuple (since the map $\Div_0(X)\to \Pic_0(X)$ is surjective (diagram (\ref{diagram of 2 short exact sequences from div to pic}))).

Similarly, we call a 4-tuple of elements in
\[
\Ho^{1}(G,\Pic(\bar{X})) \times \Ho^{1}(G,\Pic_0(\bar{X})) \times \Ho^{2}(G,\bar{K}^*/\bar{k}^*) \times \Br(X)
\]
\textit{compatible} if its elements map one to each other under the relevant maps in the middle column of (\ref{4 rows compatibility}). By Lemma \ref{The image of phi is in the image of delta}, every element of $\Br(X)$ can be completed to a compatible 4-tuple. By the commutativity of (\ref{factoring j, daigram}), composing the maps in the middle column of (\ref{4 rows compatibility}) is equivalent to applying the homomorphism $h$, namely we have a commutative diagram: 
\[
    \begin{tikzcd} [column sep=small]
    \Br(X) \arrow[r, "\phi"'] \arrow[rrr, "h", bend left=10] &
    \Ho^{2}(G,\bar{K}^*/\bar{k}^*) &
    \Ho^{1}(G,\Pic_0(\bar{X})) \arrow[l,"\delta"] \arrow[r, two heads] &
    \Ho^{1}(G,\Pic(\bar{X})).
    \end{tikzcd}
\]

Now given elements of $\Pic_0(X)$ and $\Br(X)$, we complete them into compatible 4-tuples as above, and then use the commutativity of (\ref{4 rows compatibility}) to conclude that (\ref{compatibility of rho and psi}) commutes, as desired.
\end{proof}

\end{document}